\documentclass[a4paper,11pt]{amsart}

\pdfoutput=1

\usepackage[text={400pt,660pt},centering]{geometry}

\usepackage{amsthm, amssymb, amsmath, amsfonts, mathrsfs}
\usepackage{mathtools}
\usepackage{scalerel} 
\usepackage{stackrel}
\usepackage{bbm}
\usepackage[colorlinks=true, pdfstartview=FitV, linkcolor=blue, citecolor=blue, urlcolor=blue,pagebackref=false]{hyperref}
\numberwithin{equation}{section}
\usepackage[export]{adjustbox} % For images align
\usepackage{caption}
\usepackage{subcaption}
\usepackage{enumitem}

\newtheorem{proposition}{Proposition}
\newtheorem{theorem}[proposition]{Theorem}
\newtheorem{lemma}[proposition]{Lemma}
\newtheorem{corollary}[proposition]{Corollary}

\theoremstyle{remark}
\newtheorem{remark}[proposition]{Remark}

\theoremstyle{definition}
\newtheorem{definition}[proposition]{Definition}

\numberwithin{equation}{section}
\numberwithin{proposition}{section}
\numberwithin{table}{section}

\renewcommand{\le}{\leqslant}
\renewcommand{\ge}{\geqslant}
\renewcommand{\leq}{\leqslant}
\renewcommand{\geq}{\geqslant}

\renewcommand{\subset}{\subseteq}

\newcommand{\D}{\mathcal{D}}
\newcommand{\E}{\mathbb{E}}

\renewcommand{\a}{\mathbf{a}}

\renewcommand{\d}{\mathbf{d}}

\newcommand{\ab}{{\overbracket[1pt][-1pt]{\a}}}

\newcommand{\Ll}{\left}
\newcommand{\Rr}{\right}

\newcommand{\rhs}{right-hand side}

\newcommand{\R}{\mathbb{R}}

\newcommand{\N}{\mathbb{N}}

\newcommand{\Zd}{{\mathbb{Z}^d}}
\renewcommand{\P}{\mathbb{P}}

\renewcommand{\tilde}{\widetilde}

\renewcommand{\epsilon}{\varepsilon}
\renewcommand{\d}{{\mathrm{d}}}

\newcommand{\cu}{{\scaleobj{1.2}{\square}}}

\newcommand{\p}{\mathfrak{p}}

\newcommand{\oO}{\mathcal{O}}

\newcommand{\B}{\bar{B}}

\newcommand{\T}{\mathcal{T}}
\newcommand{\var}{\mathbb{V}\!\mathrm{ar}}
\newcommand{\CFH}{\mathsf{CFH}} % Coupling in finite horizon
\newcommand{\CIH}{\mathsf{CIH}} % Coupling in infinite horizon

\newcommand{\cltf}{\mathscr{C}_\infty}
\newcommand{\Ind}[1]{\mathbf{1}_{\left\{#1\right\}}}
\newcommand{\Lip}{1\text{-}\mathrm{Lip}}

\title[Strong approximation for random walks on percolation]{Coupling between Brownian motion and random walks on the infinite percolation cluster}
\author{Chenlin Gu, Zhonggen Su, Ruizhe Xu}

\address[Chenlin Gu]{Yau Mathematical Sciences Center, Tsinghua University, Beijing, China}
\email{gclmath@tsinghua.edu.cn}

\address[Zhonggen Su]{School of Mathematical Sciences, Zhejiang University, Hangzhou, China}
\email{suzhonggen@zju.edu.cn}

\address[Ruizhe Xu]{School of Mathematical Sciences, Zhejiang University, Hangzhou, China}
\email{12135028@zju.edu.cn}
\begin{document}

	\begin{abstract}
		For the supercritical  Bernoulli bond percolation  on $\mathbb{Z}^d$ ($d \geq 2$), we give a coupling between the random walk on the infinite cluster and its limit Brownian motion, such that the maximum distance between the paths during $[0,T]$ has a mean of order $T^{\frac{1}{3}+o(1)}$. The construction of the coupling utilizes the optimal transport tool. The analysis mainly relies on local CLT and the concentration of the cluster density.
		This partially answers an open question posed by Biskup [Probab. Surv., 8:294-373, 2011].  As a direct application, our result recovers the law of the iterated logarithm proved by Duminil-Copin [arXiv:0809.4380], and further identifies the limit constant.
	\end{abstract}
	
	\maketitle
	\section{Introduction}
	\subsection{Strong approximation}
	Suppose that $\{X_n\}_{n\ge 1}$ is a sequence of i.i.d. random variables with $\E [X_1]=0, \E [X_1^2]=1$ and set $S_n :=X_1+X_2+\dots+X_n$ for $n\ge 1$.
	The celebrated Donsker's invariance principle states that the normalized partial sum process $S_n$  converges weakly to standard Brownian motion.
	
	It is natural to further ask whether a stronger convergence result can be obtained. Specifically, we wonder how to construct a Brownian motion $(B_t)_{t\ge 0}$ and the random walk $(S_k)_{k \in \N}$ in a sufficiently large probability space, with the goal of minimizing the approximation rate given by
	\begin{equation*}
		\max_{1\le k\le n}\vert S_k-B_k\vert.
	\end{equation*}
	This problem is known as \emph{embedding} or \emph{strong approximation}. The exploration of this topic can be traced back to the foundational work of Skorokhod \cite{skorohod1965} and Strassen \cite{strassen}, who obtained an approximation rate of order $O\Ll(n^{1/4}(\log n)^{1/2}(\log\log n)^{1/4}\Rr)$ assuming a finite fourth moment. 
	Later, Koml\'os, Major and Tusn\'ady \cite{KMT1} developed a completely different technique, now known as \emph{the KMT coupling},  which establishes an approximation rate of order $o(n^{1/p})$ under the assumption of finite $p$-th moments; moreover, when finite exponential moments exist, an order of $O(\log n)$ can be attained.
	Results from \cite{bartfai1966bestimmung} (see also \cite{Zaitsev2002Estimates}) proved that the $O(\log n)$ rate is optimal unless the distribution of $X_1$ is already standard normal. 
	Furthermore, Zaitsev established the KMT coupling in $\R^d$ and provided a sharp estimate in \cite{zauitsev1996, zaitsev1998}.
	More recently, Chatterjee offered a new proof of the KMT result using \emph{Stein's method} in \cite{chatterjee12}.
	
	%We note that it is not guaranteed that every Brownian motion defined on the underlying space will satisfy the aforementioned orders of approximation. This is why it is necessary to construct a new version of the original process.
	
	A comprehensive treatment of strong approximations in the i.i.d. case
	can be found in the monograph by Cs\"org\H o and R\'ev\'esz \cite{csorgo1981strong}. The theory has since been extended to encompass non-identically distributed variables \cite{sakhanenko1984rate, shao1995strong}, martingale differences \cite{monrad1991nearby, morrow1982invariance}, mixing sequences \cite{shao1993almost}, and random walk bridges \cite{dimitrov2021}. All of these results are powerful tools in probability and statistics. See Ob{\l}\'oj \cite{surveySkorokhod} for a nice survey, and see \cite{csorgo1981strong,csorgo1984KMT,shorack1986empirical,zhang1997strong} for a wide range of applications in diverse research fields.

	This paper aims to establish the first strong approximation result for random walks on the infinite percolation cluster.

	\subsection{Random walks on infinite percolation cluster}
	Percolation was  first introduced by  Broadbent and Hammersley \cite{broadbent1957percolation} as a model for the disorder medium, and has been extensively studied over the past decades to explain phase transitions  in statistical physics. For a historical overview and rigorous mathematical treatments of percolation, the reader is referred to Grimmett \cite{grimmett1999} and Kesten \cite{kesten1982percolation}.

	We focus on \emph{the  Bernoulli bond percolation} in this paper. Let $\mathbb{Z}^d$ be the $d$-dimensional Euclidean lattice.  For ${x, y \in \mathbb{Z}^{d}}$, if they are the nearest neighbors, we denote by $x \sim y$. The set $\mathcal{E}_{d}:=\left\{\{x, y\}: x, y \in \mathbb{Z}^{d}, x \sim y \right\}$ represents the set of the neighbor edges of $\mathbb{Z}^d$. For \(\p \in [0,1]\), the triplet \((\Omega, \mathcal{F}, \mathbb{P}_\p)\) stands for the probability space of the \(\mathbb{Z}^d\)-Bernoulli bond percolation: let \(\Omega = \{0,1\}^{\mathcal{E}_{d}}\) be the sample space, and let \(\omega := \{\omega(e)\}_{e \in \mathcal{E}_{d}} \in \Omega\) represent a configuration of percolation. Here, \(\omega(e) = 1\) indicates that the edge \(e\) is open, while \(\omega(e) = 0\) signifies that it is closed. The notation \(\mathcal{F}\) denotes the \(\sigma\)-algebra of subsets of \(\Omega\) generated by finite-dimensional cylinders, and we equip it with an i.i.d. Bernoulli measure of parameter \(\p\) such that
	\begin{align}\label{eq.defperco}
		\forall e \in \mathcal{E}_{d}, \qquad  \mathbb{P}_ \p(\omega(e) = 1) = 1 - \mathbb{P}_\p(\omega(e) = 0) =  \p.
	\end{align}
	For every $x, y \in \mathbb{Z}^{d}$, we denote by $x \stackrel{\omega}{\longleftrightarrow} y$ if there exists an open path connecting $x$ and $y$. Additionally, we write $x \stackrel{\omega}{\longleftrightarrow} \infty$ if $x$ belongs to an open path of infinite length. The connectivity probability is then defined as
	\begin{align}\label{eq.defconnectivity}
		\theta(\p) :=  \P_\p(x \stackrel{\omega}{\longleftrightarrow} \infty).
	\end{align}
	A connected component of open edges is called a \emph{cluster}. The infinite cluster, denoted by $\mathscr{C}_{\infty}(\omega)$ (with $\mathscr{C}_{\infty}$ being its shorthand), is defined as
	\begin{equation*}
		\mathscr{C}_{\infty}(\omega) :=\left\{x \in \mathbb{Z}^{d}: x \stackrel{\omega}{\longleftrightarrow} \infty\right\}.
	\end{equation*}
	For $d \geq 2$, Broadbent and Hammersley proved in \cite{broadbent1957percolation} the existence of the threshold $\p_c(d) \in (0,1)$ for the phase transition of connectivity. In the subcritical regime  $\p<\p_c(d)$, every vertex is almost surely in a finite open cluster. In the supercritical regime  $\p \in (\p_c(d), 1]$, we have $\theta(\p) > 0$. Aizenman, Kesten and Newman proved in \cite{aizenman1987uniqueness} almost surely there exists a unique infinite cluster $ \mathscr{C}_{\infty}(\omega)$; see also \cite{burton1989density}.
	
	Our interest is in the random walk in the supercritical regime ($d \geq 2$). Specifically, we study the \emph{variable-speed random walk} (VSRW): given a configuration $\omega$, let $(S_t)_{t\ge 0}$ start from $y \in \cltf(\omega)$, and every edge $e \in \mathcal{E}_{d}$ rings independently with an exponential clock of rate $\omega(e)$, then the random walk crosses the edge that rings first and contains the current position. VSRW is thus a Markov jump process associated with the generator
	\begin{align}\label{eq.VSRW}
		\mathcal{L}_{V} f(x):=\sum_{z: z \sim x} \omega (\{x, z\})(f(z)-f(x)).
	\end{align}

	In supercritical regime, the geometry of $\cltf$ is studied in \cite{penrose1996large, pisztora1996surface}, and is proved to be close to  $\Zd$ in large scale. %Nevertheless, the holes perturb the random walk, so techniques are required to handle the randomness of the graph. 
	Therefore, their associated random walks should behave similarly as well. 
	The \emph{quenched} invariance principle for the discrete-time simple random walk on supercritical percolation was initially obtained by Sidoravicius and Sznitman in \cite{sidoravicius2004quenched} for $d \geq 4$, then generalized by Berger and Biskup in \cite{berger2007quenched}, as well as by Mathieu and Piatnitski in \cite{mathieu2007quenched} for all $d \geq 2$. Slightly adapted from their results, the following quenched invariance principle is stated for VSRW: for $\P_\p$-almost every configuration $\omega$, we have the weak convergence under the Skorokhod topology
	\begin{equation}\label{equ:QIP}
		(\epsilon S_{t/\epsilon^2})_{t \geq 0}  \stackrel[\epsilon \to 0]{d}{\longrightarrow} (\bar{\sigma}B_t)_{t \geq 0}.
	\end{equation}
	Here, $(B_t)_{t\ge 0}$ denotes the standard $d$-dimensional Brownian motion, and the diffusive constant $\bar{\sigma}>0$ is independent of configuration $\omega$. 
	
	The invariance principle also holds for \emph{the random conductance model}, which covers the percolation. This illustrates the robustness of Brownian universality; see \cite{ABDH,andres2015invariance,  BD,hambly2009parabolic, procaccia2016quenched, sapozhnikov2017random, chen2024quenched} for examples and the survey \cite{biskup2011recent} by Biskup.
	
	In the proof of the results above, one crucial step is reducing the random walk to a martingale  via \emph{the corrector method}. However, to establish the strong approximation for this martingale requires a detailed analysis of the quadratic variation, which is highly nontrivial. In \cite[Page~346, Section~4.4]{biskup2011recent}, Biskup highlighted this issue: \textit{``An important open problem concerns the rate of convergence and quantification of errors in martingale approximations.''} Therefore, the strong approximation for random walks on percolation was missing in the literature for long time.

	The random walk on random conductance is connected to homogenization theory, which offers many valuable insights, including the corrector method mentioned earlier. This connection has been recognized since the work of Kozlov \cite{kozlov1985}, Papanicolaou, and Varadhan \cite{varadhan1982}; see also the discussion in \cite[Section~3.2, Section~6.1]{biskup2011recent}. However, the quantification of stochastic homogenization was challenging, and not resolved until the emergence of the recent progress: see \cite{armstrong2016quantitative, armstrong2016lipschitz, armstrong2016mesoscopic, armstrong2017additive, NS, vardecay, gloria2011optimal, gloria2012optimal, gloria2015quantification, GO3} and the monograph \cite{AKMbook, armstrong2022elliptic}. Subsequently, quantitative homogenization theory was applied to supercritical percolation in \cite{armstrong2022elliptic, dario2021corrector, dario2021quantitative, gu2022efficient, bou2023rigidity}, resulting in a suite of analytical tools now available, including the Liouville theorem, optimal corrector growth, local CLT, and rigidity of harmonic functions, among others.
	
	Thanks to all these developments, we are ready to present the first strong approximation for the random walk on supercritical percolation. The novelty of our approach lies in integrating homogenization techniques into an optimal transport framework, enabling us to construct a quantitative coupling through a relatively elementary approach. Such a strong approximation is also regarded as a key tool for further exploration of related topics, including mixing time and the law of the iterated logarithm; see \cite[Section~3.1]{benjamini2003} and the discussion below \cite[Theorem~1.2]{kumagai2016}, respectively.
	
	\subsection{Main results}
	Our major contribution in this paper is constructing a coupling between VSRW in \eqref{eq.VSRW} and its limit Brownian motion in \eqref{equ:QIP}, with an upper bound for the maximum difference. For simplicity, we will use shorthand $\bar{B}_t:=\bar{\sigma}B_t$ in the rest of the paper. To state the result, we need to introduce an $\omega$-measurable variable defined for every $\delta > 0$ and $y \in \Zd$ (see \eqref{eq.defTy} for its detailed definition) 
	\begin{align}\label{eq.defT}
		\T_\delta(y) : \Omega \to [0,\infty) , 
	\end{align}
	which represents a minimal scale for ``a good local configuration" to ensure the validity of the statement. It satisfies the stretched exponential tail estimate
	\begin{align}\label{eq.TailT}
		\forall T>0,\quad \P_\p\left(\mathcal{T}_\delta(y)\geq T\right)\le C\exp\left(-\frac{T^s}{C}\right),
	\end{align}
	where $C,s$ are all finite positive constants depending on $\mathfrak{p},d,\delta$. The notation $\vert \cdot\vert_2$ denotes the Euclidean distance in $\mathbb{R}^d$ in the statement.
	
	\begin{theorem}\label{The:main_result}
		Fix  $d\ge 2$ and $\p \in (\p_c(d), 1]$. For almost every configuration $\omega\in \Omega$ and $y\in\cltf(\omega)$, we can construct a version of VSRW $(S_t)_{t \geq 0}$
		and Brownian motion $(\B_t)_{t \geq 0}$ with diffusive constant $\bar{\sigma}$ in the same probability space, both starting from $y$, such that for every $\delta >0$ and for all $T > \T_\delta(y)$, we have
		\begin{align}\label{eq.main}
			\E^\omega\Big[\sup_{t\in [0,T]} \vert S_t - \B_t \vert_2 \Big] \leq K T^{\frac{1}{3}+\delta},
		\end{align}
		where $K$ is a finite positive constant depending on $d, \p,\delta$.
	\end{theorem}

	We emphasize that both the coupling $(S_t, \B_t)_{t \geq 0}$ and the collection of minimal scales $\{\T_\delta(y)\}_{y \in \Zd, \delta > 0}$ depend only on the sample of the percolation $\omega$, but they do not depend on each other. For different $\delta > 0$ and starting point $y \in \cltf(\omega)$, the random walk requires an exploration longer than $\T_\delta(y)$ to achieve the estimate \eqref{eq.main}. Since $\delta$ can be made arbitrarily small, the rate can be considered as $T^{\frac{1}{3} + o(1)}$ as $T \to + \infty$. 
	
	\bigskip
	
	This result provides a strong tool to analyze the of path property of VSRW. For example, 
	Theorem~\ref{The:main_result} allows us to recover the invariance principle \eqref{equ:QIP}.  
	Actually, we can prove Theorem~\ref{The:main_result} without any knowledge about the invariance principle. In particular, there is an alternative characterization of the diffusive constant from \cite{armstrong2018elliptic} and \cite{dario2021quantitative}. The homogenized matrix 
	$\ab$ is defined in \cite[Definition~5.1]{armstrong2018elliptic} using a variational formula. It is a scalar matrix, given by
	\begin{align}\label{eq.defa}
		\ab= C^*\mathrm{Id},
	\end{align}
	where $C^*$ is a positive constant. See Lemma~\ref{lem.scalar} for a self-contained proof. Subsequently, \cite[eq.(181)]{dario2021quantitative} defines the diffusive constant 
	$\tilde{\sigma}^2$ as follows:
	\begin{align}\label{eq.defSigma2} 
		\tilde{\sigma}^2 := 2 \theta(\p)^{-1}C^*,
	\end{align}
	where $C^*$ is given in \eqref{eq.defa}.
	See also the discussion \cite[Remark~7, Remark~19]{dario2021quantitative}. All the elements in the proof Theorem~\ref{The:main_result} are based on the definition \eqref{eq.defSigma2}. Therefore, we actually establish a strong approximation between $(S_t)_{t\ge 0}$ and $(\tilde{\sigma} B_t)_{t \geq 0}$, leading to the invariance principle 
	\begin{equation}
		(\epsilon S_{t/\epsilon^2})_{t \geq 0} \stackrel[\epsilon \to 0]{d}{\longrightarrow} (\tilde{\sigma}B_t)_{t \geq 0}.
	\end{equation}
	Consequently, the definition $\tilde{\sigma}$ in \eqref{eq.defSigma2} coincides with $\bar{\sigma}$ in %\cite{sidoravicius2004quenched, berger2007quenched, mathieu2007quenched} 
	\eqref{equ:QIP} using the uniqueness of weak convergence limits, i.e., $\tilde{\sigma} = \bar{\sigma}$. 
	
	Another interesting application of Theorem~\ref{The:main_result} is the law of the iterated logarithm (LIL) of random walks on percolation cluster.
	In this direction, Duminil-Copin established the  LIL in \cite{duminil2008law} for \emph{constant-speed random walk} (CSRW), which is a continuous time Markov jump process with    generator
	\begin{align}\label{eq.CSRW}
		\mathcal{L}_{C} f(x):=\omega_x^{-1}\sum_{z: z \sim x} \omega (\{x, z\})(f(z)-f(x)).
	\end{align}
	Here $\omega_x:=\sum_{z:z\sim x}\omega\left(\{x,z\}\right)$ is the degree of vertex $x$. The proof in \cite{duminil2008law} utilized the heat kernel estimate by Barlow \cite{barlow2004random} along with the ergodic theory. A similar argument was later implemented by Kumagai and Nakamura in \cite{kumagai2016}, where the LIL was generalized to a family of random walks on random conductance models satisfying heat kernel estimates.
	
	However, in all these work, the limit constants in LILs are not stated explicitly because the arguments rely on the zero-one law. 
	Kumagai and Nakamura, in discussion with Biskup, have pointed out another natural approach to derive the LIL using couplings, as mentioned in the paragraph around \cite[Theorem~1.2]{kumagai2016}: \textit{``... if the random walk can be embedded into Brownian motion in some strong sense..., then (1.6),(1.7) can be shown as a consequence ([8]). It would be very interesting to prove such a strong approximation theorem.''} Theorem~\ref{The:main_result} now realizes this idea with a more direct proof and identifies the limit constant. 
	
	It is important to note that the limit constant in the LIL depends not only on the type of random walk (VSRW or CSRW), but also on the choice of the vector $\ell_p$-norm $\vert \cdot \vert_p$ (see \eqref{eq.lp}). More precisely, we have Corollary \ref{Cor:LIL_constant}.
	
	\begin{corollary}\label{Cor:LIL_constant}
		Fix  $d\ge 2$ and $\p \in (\p_c(d), 1]$. For every $p \in [1,\infty)$, for almost every configuration $\omega\in\Omega$, and for all $y\in\mathscr{C}_\infty(\omega)$, the VSRW $(S_t)_{t\ge 0}$ starting from $y$ satisfies the LIL
		\begin{align}\label{eq.LIL_VSRW}
			P^\omega_y\Big(\limsup_{t\to\infty}\frac{|S_t|_p}{\sqrt{2t\log\log t}}= \gamma_{d,p} \bar{\sigma} \Big)=1,
		\end{align}
		and the CSRW $(\tilde{S}_t)_{t\ge 0}$ starting from $y$ satisfies the LIL
		\begin{align}\label{eq.LIL_CSRW}
			P^\omega_y\Big(\limsup_{t\to\infty}\frac{|\tilde{S_t}|_p}{\sqrt{2t\log\log t}}=\alpha_{d,\p} \gamma_{d,p} \bar{\sigma}\Big)=1.
		\end{align}
		Here, $\bar{\sigma}$ is the diffusive constant in \eqref{equ:QIP}, and constants $\gamma_{d,p}, \alpha_{d,\p}$ are defined as
		\begin{align}\label{eq.defLIL_para}
			\gamma_{d,p} := \max\Ll\{d^{\frac{1}{p}-\frac{1}{2}}, 1\Rr\}, \qquad \alpha_{d,\p} := \big(2d\E_\p[\omega\left(\{0,x\}\right)|0\in\cltf]\big)^{-\frac{1}{2}},
		\end{align}
		where $x$ is a neighboring vertex of $0$.
	\end{corollary}
	
	\begin{remark}
		Duminil-Copin studied the result for CSRW under $\vert \cdot \vert_1$ in \cite{duminil2008law}. According to Corollary~\ref{Cor:LIL_constant}, we have $\gamma_{d,1} = d^{\frac{1}{2}}$ which cancels out $d^{-\frac{1}{2}}$ in $\alpha_{d,\p}$. Notice the minor difference in the the scaling factor of \cite[Theorem 1.1]{duminil2008law}, so the limit constant $c(d,\p)$ there should be 
		\begin{align*}
			c(d,\p) = \sqrt{2}\alpha_{d,\p} \gamma_{d,1} \bar{\sigma} = \big(\E_\p[\omega(\{0,x\})|0\in\cltf]\big)^{-\frac{1}{2}} \bar{\sigma}.
		\end{align*}
	\end{remark}

	\subsection{Ingredients of the coupling}\label{subsec.ingredient}
	In this part, we outline the construction of the coupling in three steps. The analysis of the approximation is also discussed without delving into technical details.
	
	\textit{\textbf{Step~1: coupling of marginal distributions.}} The coupling between the marginal distributions of $S_t$ and $\bar{B}_t$ serves as a cornerstone of our analysis. The weak invariance principle indicates that they should be close for large $t$.

	The main challenge is that the supports of the marginal distributions are mutually singular, thus we need to take the geometry of $\cltf$ into consideration. Our solution turns out to be the $1$-Wasserstein distance, denoted by $W_1(\cdot, \cdot)$. This distance is robust for mutually singular measures, and it naturally yields a coupling that achieves $W_1(S_t, \B_t)$; see \cite[Theorem~4.1]{Villani2009Optimal}.
	
	The $1$-Wasserstein distance has a nice characterization via  \emph{the Kantorovich duality theorem}, which provides a tool to calculate $W_1(S_t, \B_t)$
	\begin{align}\label{eq.dualSB}
		W_1(S_t, \B_t) =
		\sup_{f\in \Lip(\mathbb{R}^d)}\Bigg\vert \sum_{x\in \mathscr{C}_\infty}f(x)p^\omega(t,x,y)-\int_{\mathbb{R}^d}f(x)\bar{p}(t,x - y)\, \d x\Bigg\vert.
	\end{align}
	Here, $p^\omega(t,\cdot,y)$ is the transition probability for $S_t$ starting from $y \in \cltf$, and $\bar{p}(t, \cdot)$ is the density for Brownian motion $\B_t$. Two technical inputs to further analyze \eqref{eq.dualSB} are local CLT estimate in \cite[Theorem~2.1]{dario2021quantitative} and concentration inequality of cluster density in \cite[Proposition 14]{dario2021quantitative}.  In combination, we obtain a coupling $(S_t,\B_t)$ such that, for every arbitrarily small $\delta > 0$,
	\begin{equation}\label{eq.W1marginal}
		\E^\omega\Ll[ \vert S_t-\B_t\vert_2 \Rr] = W_1(S_t, \B_t) \le Ct^\delta,
	\end{equation}
	where $t$ is larger than an $\omega$-measurable random variable (see \eqref{equ:random_time} below).

	\textit{\textbf{Step~2: coupling of the process in finite horizon.}} In this step, we implement the marginal coupling to some discrete time points first, then further extend it to the entire interval [0,T]. See Figure \ref{pic.coupling} for an illustration.
	\begin{itemize}[label=---]
		\item  Divide the interval $[0,T]$   into $n$ segments of length $\Delta T$, and set
		\begin{align*}
			\Delta T := T/n, \quad t_k := k \Delta T,  \quad I_k = (t_k, t_{k+1}),	\qquad k\in\{0,1,\dots, n-1\}.
		\end{align*}
		\item   For every $k = 0, \cdots, n-1$ and conditioned on the position of $S_{t_k}$,  we sample the increment $(\Delta S_{t_k}, \Delta \B_{t_k})$ as a coupling of $\big(p^\omega(\Delta T,S_{t_k} +\cdot,S_{t_k}), \bar{p}(\Delta T,\cdot) \big)$ following Step~1. By gluing them together, 
		\begin{align*}
			(S_{t_{k+1}}, \B_{t_{k+1}}) := (S_{t_k}, \B_{t_{k}}) + (\Delta S_{t_k}, \Delta \B_{t_{k}}),
		\end{align*}
		the coupled process $(S_{t_k}, \B_{t_k})_{0 \leq k \leq n}$ is obtained.
		\item   Sample the Markov process for $(S_t)_{t \in [0,T]}$ following the law of VSRW conditioned on $(S_{t_k})_{0\le k\le n}$, and sample $(\B_t)_{t \in [0,T]}$ as the Brownian motion of diffusive constant $\bar{\sigma}^2$  conditioned on $(\B_{t_k})_{0\le k\le n}$. 
	\end{itemize}
	Heuristically speaking,    every coupling of increment $(\Delta S_{t_k}, \Delta \B_{t_k})$ will accumulate an error of order $(\Delta T)^\delta$  as claimed in \eqref{eq.W1marginal}. The fluctuation over each $I_k$ is of order $\sqrt{\Delta T}$ with a sub-exponential tail, thus the maximum fluctuation of the total $n$ intervals is of order $(\log n) \sqrt{\Delta T}$. Therefore, we get an estimate
	\begin{align*} %\label{eq.ErrorRough}
		\E^\omega \Big[\sup_{t\in[0,T]}\vert S_t-\bar{B}_t\vert_2\Big] \leq K\Big(n (\Delta T)^\delta +  (\log n) \sqrt{\Delta T}\Big).
	\end{align*}
	With the optimization of parameter $\Delta T = T/n$, we choose $n = \lfloor T^{\frac{1}{3}} \rfloor$  to obtain an approximation of order $T^{\frac{1}{3}+\delta}$.
	
	\begin{figure}[t]
		\centering
		\includegraphics[height=120pt]{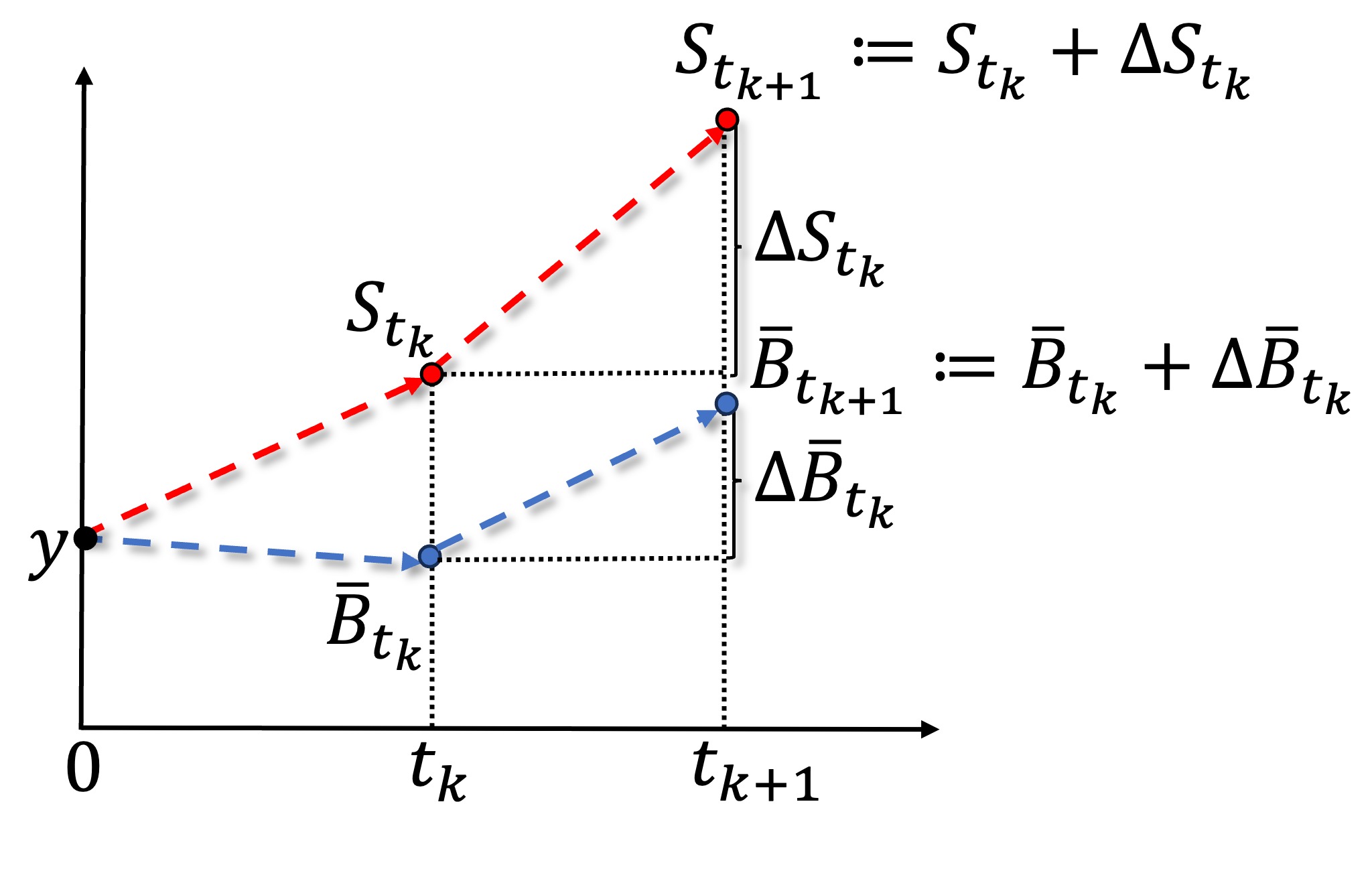}
		\includegraphics[height=120pt]{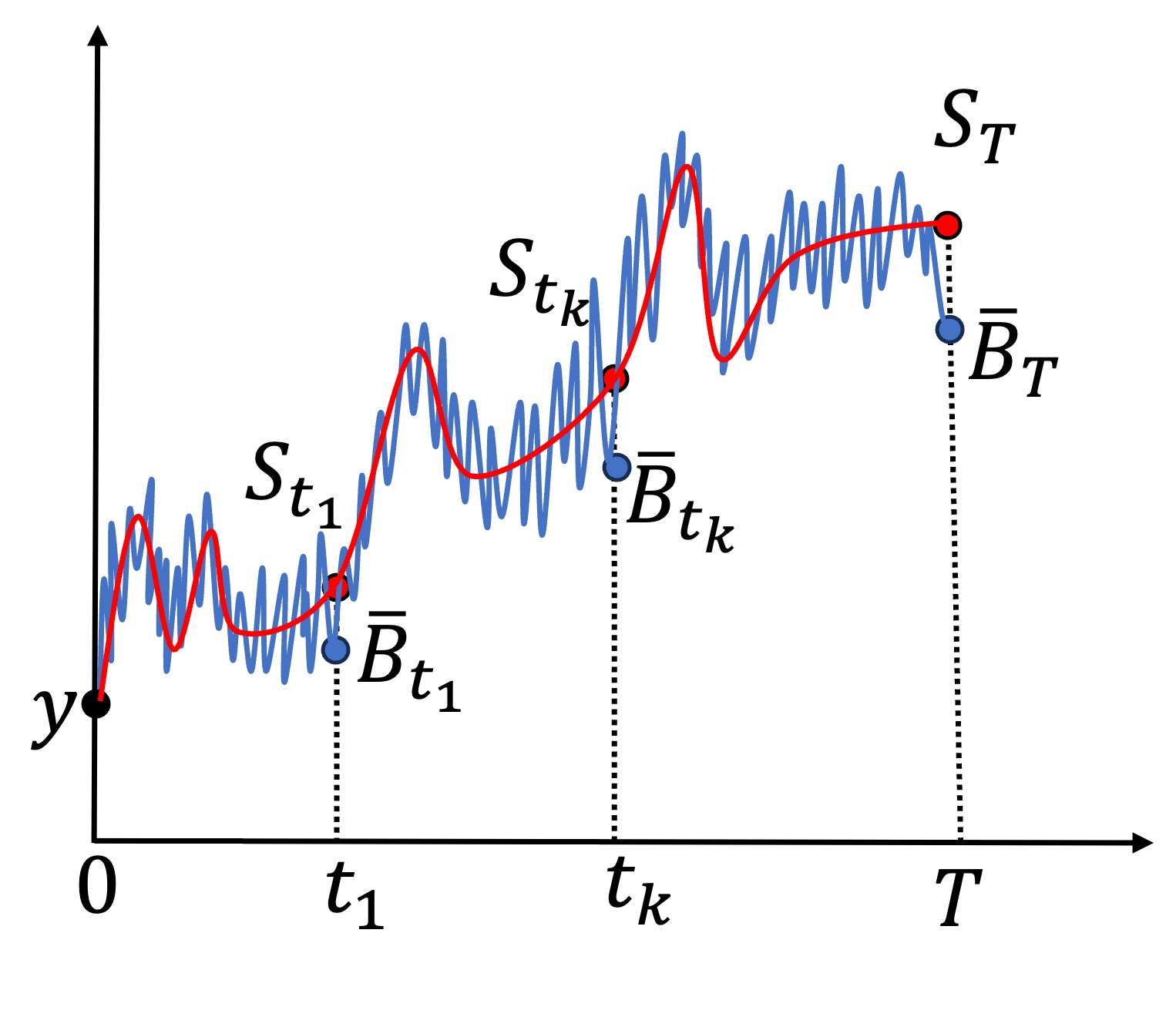}
		\caption{the construction of the coupling $(S_t, \B_t)_{t \in [0,T]}$.}\label{pic.coupling}
	\end{figure}
	
	\textit{\textbf{Step~3: coupling of the process in infinite horizon.}} We make a triadic decomposition of time $\R_+$ with endpoints $(T_k)_{k \in \N_+}$, and implement the coupling in Step~2 over each interval $[T_k, T_{k+1}]$. This gives a coupling over the entire positive real line that achieves the desired order in \eqref{eq.main}.

	\subsection{Outlooks}
	The arguments presented in this paper are quite robust and can be adapted to obtain strong approximations for other random walks on random conductance as long as we have the following estimate 
	\begin{align*} %\label{eq.W1Minimum}
		W_1(S_t, \B_t) = o(t^{\frac{1}{2}}).
	\end{align*}
    
     It is natural to ask about the optimality of the coupling. To the best of our knowledge, no conjectures about the optimal rate have been formally proposed in the literature. Currently, the KMT coupling stands as a promising candidate for achieving such optimality. Actually, one can still insert Step~1 above into the KMT coupling for its marginal distributions, while the main issue is the analysis of the error. If for sufficiently large $t$, \eqref{eq.W1marginal} could be strengthened to
    \begin{align} \label{eq.W1Conjecture}
	   W_1(S_t, \B_t) = O(1),
    \end{align}
	then gluing the marginal distributions through a dyadic way would yield the error term $O(\log T)$, matching the optimal rate in the classical KMT result.
    The conjecture \eqref{eq.W1Conjecture} is supported by multivariate CLT approximations. Significant progress has been made in this direction through a series of works \cite{zhai2018, EMZ2020, CFP19, thomas2020}, which demonstrate that under various settings, the sequence of normalized i.i.d. random vectors $\{X_n\}_{n\ge 1}$ satisfies
    \begin{align*}
        W_1\Ll(n^{-1/2}\sum_{i=1}^n X_i, \mathcal{N}(0,1)\Rr) = O(n^{-1/2}).
    \end{align*}
	After scaling, this result precisely matches the rate predicted by conjecture \eqref{eq.W1Conjecture}.

        \iffalse{It is then natural to ask the optimality of coupling. To the best of our knowledge, no conjecture was claimed about the optimal rate in the literature, but the KMT coupling can be one object. For the moment, let us point out a bottleneck here. Actually, one can still insert Step~1 above into the KMT coupling for its marginal distributions, while the main issue is the analysis of the error. However, \eqref{eq.W1marginal} is sub-optimal, so it will always give a polynomial error bound, which is far from the the sharp bound for KMT coupling. We conjecture the following estimate
    	\begin{align*} %\label{eq.W1Conjecture}
    		W_1(S_t, \B_t) = O(1),
    	\end{align*}
    	for $t$ large enough. This is because, for the i.i.d. real valued random variables $(X_i)_{i \in \N}$ such that $\E[X_i] = 0$ and $\var[X_i] = 1$, the work \cite{Agn57, Esseen58} proved that 
    	\begin{align*}
    		W_1\Ll(n^{-1/2}\sum_{i=1}^n X_i, \mathcal{N}(0,1)\Rr) = O(n^{-1/2}),
    	\end{align*}
    	and Rio justified the sharpness of the estimate in \cite[Theorem~5.1]{Rio09}. Concerning the multivariate CLT approximation, there was a series of development in \cite{valiant2011, zhai2018, EMZ2020, CFP19}, and the optimal rate was obtained by Bonis in \cite{thomas2020}. Therefore, after scaling,  the counterpart of \eqref{eq.W1marginal} for the standard centered random walk is $O(1)$. This order is also consistent with the classical KMT coupling, which glues the marginal distributions in a dyadic way and leads to the typical error $\log n$ between paths of length $n$. }\fi

	Finally, another natural direction for future research would be to extend the study of marginal coupling in the more general setting of $p$-Wasserstein distance. Though $p$-Wasserstein distances may offer additional nice structures, such an extension also presents some challenges. In particular, the trade-off between dimension and integrability in transport cost must be carefully balanced (see \cite{AKT84, Talagrand94, HHP06, FG15}).

	\subsection{Organization of the paper}
	The rest of the paper is organized as follows. In Section~\ref{sec.pre}, we will introduce some basic notations and inequalities that will be used throughout the paper.
	Section~\ref{Wasserstein} reviews fundamental concepts of optimal transport, and employs them to establish a coupling of marginal distributions between $S_t$ and $\B_t$.
	Section~\ref{sec.CouplingProcess} is devoted to the delicate construction of the coupled processes. Specifically, Section~\ref{Finite} focuses on the finite horizon version $\left(S_t,\bar{B}_t\right)_{t\in[0,T]}$, which is then extended to the infinite horizon version $(S_t,\B_t)_{t\ge 0}$ in Section~\ref{infinite}.
	Finally, in Section~\ref{sec.LIL}, we review several aspects of the LIL and utilize our strong approximation theorem to identify the limit constant in the context of percolation.

	\section{Preliminaries}\label{sec.pre}
	\subsection{General notations} \label{notation_percolation} 
	Let $\mathbb{R}^d$ denote the $d$-dimensional Euclidean space. For a vector $v \in \mathbb{R}^{d}$ and an integer $i \in\{1, \ldots, d\}$, we denote its $i$-th component by $v^i$, such that $v=\left(v^1, \ldots,v^d\right)^\top$. The $\ell_p$-norm, denoted by  $\vert \cdot \vert_p$, is defined for every $p \in [1, \infty)$  as follows:
	\begin{equation}\label{eq.lp}
		\forall x  = (x^1,x^2, \cdots, x^d) \in \R^d, \qquad \vert x \vert_p := \Big(\sum_{i=1}^d \vert x^i \vert^p\Big)^{\frac{1}{p}}.
	\end{equation}
	In particular, we use $\vert \cdot \vert$ as shorthand of $\vert \cdot \vert_2$.

	For every $x \in \Zd$ and $r > 0$, the notation $B(x,r) :=\{y\in \mathbb{Z}^d: \vert y-x\vert\le r\}$ stands for the discrete Euclidean ball of radius $r$ centered at $x$. 
	
	%We denote by $[n]$ the set $\{0,1,2,\dots,n-1\}$.

	\subsection{Orlicz norm}\label{notation_Orlicz}
	Let $\psi:\mathbb{R}_+\to\mathbb{R}_+$ be an increasing convex function that satisfies 
	\begin{align*}
		\psi(0)=0,\qquad \lim_{t\to\infty}\psi(t)=\infty.            
	\end{align*}
	The $\psi$-Orlicz norm of a random variable $X$ is defined as 
	\begin{equation}\label{eq.defOrlicz}
		\| X\|_\psi :=\inf\{t>0: \mathbb{E}[\psi (t^{-1}\vert X\vert)]\leq 1 \}.
	\end{equation}
	%where $\|X\|_\psi$ is infinite if there is no $t < \infty$ for which the expectation $\mathbb{E}[\psi (t^{-1}\vert X\vert)]$ exists. 
	In this paper, we use the Orlicz norm $\|\cdot\|_{\psi_s}$ defined by the convex function 
	\[
	\psi_s(u) = \begin{cases}\exp\left( \frac{1}{s}(u\vee 1)^s\right)-\exp\left(\frac{1}{s}\right) & \text { if } s\in(0,1),  \\ \exp \left(u^s\right)-1 & \text { if }s\in[1,\infty).\end{cases}
	\]
	Another related  notation $\mathcal{O}_s$ is introduced in \cite[Appendix~A]{AKMbook}, which is defined as
	\begin{equation}\label{eq.defOs}
		X \leq \mathcal{O}_{s}(\theta) \quad \text { if and only if } \quad \|X\|_{\psi_s}\leq \theta .
	\end{equation}
	The notation $\mathcal{O}_s(\cdot)$ can be viewed as a relaxed version of the Orlicz norm $\|\cdot\|_{\psi_s}$, since it provides an upper bound.
	
	\smallskip
	We summarize some useful properties of $\mathcal{O}_s$ as outlined in \cite[Appendix~A]{AKMbook}. Given a random variable $X$ such that $X \leq \mathcal{O}_{s}(\theta)$, then for every $\lambda \in \mathbb{R}_{+}$, it follows that
	\begin{equation}\label{equ:multiply}
		\lambda X \leq \mathcal{O}_{s}(\lambda \theta).
	\end{equation}
	The Markov inequality implies
	\begin{equation*}
		X\le\mathcal{O}_s(\theta)~\Longrightarrow~\forall x\ge 0, \qquad
		\P(X\ge \theta x)\le C_s\exp\left( -x^s \right),
	\end{equation*}
	and the integration by parts yields 
	\begin{equation}\label{equ:tail_1}
		\forall x\ge 0, \qquad \P(X\ge \theta x)\le \exp(-x^s)~\Longrightarrow~X\le \mathcal{O}_s\left(C_s\theta\right).
	\end{equation}
	We can reduce the stochastic integrability parameter $s$ in $X \leq \mathcal{O}_{s}(\theta)$: for each $s^{\prime} \in(0, s]$, there exists a positive constant $C_{s^{\prime}}<\infty$ such that 
	$$X \leq \mathcal{O}_{s^{\prime}}\left(C_{s^{\prime}} \theta\right).$$
	Given a collection of random variables $\{X_i\}_{i=1}^n$ and nonnegative constants $\{C_i\}_{i=1}^n$ satisfying $X_i\le\mathcal{O}_s(C_i)$, the subadditivity property of the Orlicz norm implies
	\begin{equation}\label{equ:Orlicz_sum}
		\sum_{i=1}^{n}X_i \le \mathcal{O}_s\Big(\sum_{i=1}^{n}C_i\Big).
	\end{equation}
	For a sequence of random variables $\{X_i\}_{i=1}^N$ satisfying $X_i\le \mathcal{O}_s(1)$, 	Lemma 4.3 of \cite{Gu2020uniform}  gives an estimate of its maximum that 
	\begin{equation}\label{equ:max}
		\Big(\max_{1\le i\le N}X_i\Big)\le \mathcal{O}_s\Bigg(\Big(\frac{\log (2N)}{\log(4/3)}\Big)^{\frac{1}{s}}\Bigg).
	\end{equation}
	
	The following lemma about the minimal scale is useful.
	\begin{lemma}[\cite{armstrong2018elliptic}, Lemma 2.2]\label{lem:scale}
		Fix $K \geq 1, s>0$ and $\beta>0$ and let $\left\{X_{n}\right\}_{n \in \mathbb{N}}$ be a sequence of nonnegative random variables satisfying $X_{n} \leq \mathcal{O}_{s}\left(K 3^{-n \beta}\right)$ for every $n \in \mathbb{N}$. Then, there exists a constant $C:=C(s, \beta, K)>0$ such that the random scale 
		\begin{align*}
			\mathcal{M}:=\sup \left\{3^{n} \in \mathbb{N}: X_{n} \geq 1\right\},
		\end{align*}
		satisfies the estimate 
		\begin{align*}
			\mathcal{M} \leq \mathcal{O}_{\beta s}(C).
		\end{align*}
	\end{lemma}

	\subsection{Expectation on rare event}
	Let $X$ be a nonnegative random variable and let $A$ be a rare event. A naive upper bound estimate  is $\mathbb{E}\left[X\mathbf{1}_{A}\right]\le \mathbb{E}[X]$. 
	This upper bound can be significantly improved if the tail probability of $X$ is accessible.
	\begin{lemma}\label{restrict}
		Let $X$ be a nonnegative random variable satisfying
		\begin{equation}\label{equ:decay}
			\mathbb{P}\left(X\ge y\right)\le 2\exp\left(-\frac{y}{C}\right),
		\end{equation}
		with $C > 0$. Then for every event $A$ satisfying $\mathbb{P}(A)=\epsilon > 0$, we have
		\begin{align*}
			\mathbb{E} \left[X\mathbf{1}_{A}\right]\le  C\epsilon \log (2/\epsilon)+C\epsilon.
		\end{align*}
	\end{lemma}
	
	\begin{proof}
		Using the definition
		$$\{X\mathbf{1}_{A}>y\}=\{X>y\}\cap A,$$
		we deduce that
		$$\mathbb{P}\left(X\mathbf{1}_{A}>y\right)\le \min\{\mathbb{P}(X>y), \epsilon\}.$$
		Let $K_\epsilon$ be the $\epsilon$-quantile for $X$, defined as
		$$K_\epsilon :=\sup\{y>0: \mathbb{P}(X>y)\ge \epsilon\}.$$
		From \eqref{equ:decay}, we can obtain
		$$K_\epsilon\le C\log (2/\epsilon).$$
		By the integration by parts, it follows that
		\begin{align*}
			\mathbb{E}\left[X\mathbf{1}_{A} \right] &=\int_0^{K_\epsilon}\mathbb{P}(X\mathbf{1}_{A}>y)\,\d y+\int_{K_\epsilon}^{\infty}\mathbb{P}(X\mathbf{1}_{A}> y)\,\d y\\
			&\le \epsilon K_\epsilon +\int_{K_\epsilon}^{\infty} \mathbb{P}(X> y) \,\d y\\
			&\le C\epsilon \log (2/\epsilon)+C\epsilon.
		\end{align*}
	\end{proof}

	\subsection{Conventions}
	We outline some conventions throughout the paper.
	\begin{itemize}
		\item The constants and exponents, such as $C$ and $s$, are always finite positive, and may vary between lines. We use the notation $C := C(d,\p,\delta)$ to indicate that the constant $C$ depends on the parameters $d, \p$ and $\delta$. 
		\item All minimal scales discussed in this paper, such as $\mathcal{T}_\delta(y)$ and $\mathcal{R}(y)$, are meaningful on the event $\{y \in \cltf\}$ and are assumed to be greater than $1$. Otherwise, on the event $\{y \notin \cltf\}$, they are extended trivially as $0$.
		\item We let $\P_{\p}, \P^\omega$ and $P^\omega_y$ denote the probability measures for the percolation, the coupling space under the configuration $\omega$, and the VSRW or CSRW starting from $y$, respectively. Similarly, the notations $\E_{\p}, \E^\omega$ and $E^\omega_y$ denote their corresponding expectations.
	\end{itemize}

	\section{Wasserstein distance between marginal distributions}\label{Wasserstein}
	In this section, we address the coupling of marginal distributions. We begin with some basic definitions of coupling, Wasserstein distance, and Lipschitz functions as presented in the monograph \cite[Chapter~1, Chapter~4, Chapter~6]{Villani2009Optimal}.
	\begin{definition}[Coupling]
		Let $(\mathcal{X},\mu)$ and $(\mathcal{Y},\nu)$ be two probability spaces. A coupling of $(\mu,\nu)$ is a pair of random variables $(X,Y)$ living in a common probability space $(\mathcal{X}\times \mathcal{Y}, \mathbb{P})$, such that $\mathrm{Law}(X)=\mu$ and $\mathrm{Law}(Y)=\nu$. The collection of all the couplings is then denoted by
		\begin{align*}
			\varPi (\mu,\nu):=\{\operatorname{Law}(X,Y):\operatorname{Law}(X)=\mu, \operatorname{Law}(Y)=\nu\}.
		\end{align*} 
	\end{definition}

	\begin{definition}[Wasserstein space]\label{def.Wp}
		Let $(\mathcal{X}, d)$ be a Polish metric space. For every $p \in[1, \infty)$, the $p$-Wasserstein distance between  two probability measures $\mu$ and $\nu$ is
		\begin{align*}
			W_{p}(\mu, \nu) :=\Big(\inf _{\pi \in \varPi(\mu, \nu)} \int_{\mathcal{X} \times \mathcal{X}} d(x, y)^{p} \, \d \pi(x, y)\Big)^{1 / p}.
		\end{align*}
		Furthermore, \cite[Theorem 4.1]{Villani2009Optimal} ensures the existence of a coupling $ \pi^*\in \varPi (\mu,\nu)$ to attain $W_p(\mu,\nu)$.
	\end{definition}

	\begin{definition}[Lipschitz function]
		Given $(\mathcal{X},d)$ a metric space, a function ${f: \mathcal{X}\to \mathbb{R}}$ is called $L$-Lipschitz if 
		\begin{align*}
			\forall x,y\in\mathcal{X}, \qquad \vert f(x)-f(y)\vert\le L d(x,y). 
		\end{align*}
		The family of all $1$-Lipschitz functions is denoted by $\Lip(\mathcal{X})$.
	\end{definition}
	
	We primarily utilize 1-Wasserstein distance, since it has a nice characterization via Lipschitz functions. This result is stated in the following lemma; see \cite[Theorem~5.10 and Remark~6.5]{Villani2009Optimal} for details.
	\begin{lemma}[Kantorovich duality]\label{duality}
		Let $(\mathcal{X},d)$ be a Polish space. Then, for every two probability measures $\mu,\nu\in\mathcal{P}(\mathcal{X}):=\{\rho:\int_\mathcal{X} d(\cdot,x) \rho(\d x)<\infty\}$, we have
		\begin{align*}
			W_1(\mu,\nu)=\inf_{\pi\in \varPi (\mu,\nu)}\int_{\mathcal{X} \times \mathcal{X} }d(x,y)\,\d\pi(x,y)=\sup_{f\in \Lip(\mathcal{X})}\left \vert \int_\mathcal{X} f \,\d\mu-\int_\mathcal{X} f \,\d\nu\right\vert.
		\end{align*}
		
	\end{lemma}
	
	The main result in this section is Proposition~\ref{W_1} below, which provides an upper bound on the 1-Wasserstein distance between the distributions of $S_t$ and $\B_t$. The transition density of $\B_t$ follows a Gaussian distribution with diffusive constant $\bar{\sigma}^2$
	\begin{align*}
		\bar{p}(t,x-y):=\frac{1}{(2\pi \bar{\sigma}^2t)^{d/2}}\exp\left(-\frac{\vert x-y\vert^2}{2\bar{\sigma}^2t}\right),\qquad x,y\in \mathbb{R}^d.
	\end{align*} 
	We define the transition probability of the random walk $(S_t)_{t \geq 0}$ starting from $y$ as
	\begin{equation*}
		p^{\omega}(t, x, y):= P_{y}^{\omega}\left(S_{t}=x\right),\qquad x,y\in \cltf.
	\end{equation*}

	\begin{proposition}\label{W_1}
		For every $\delta>0$ and $y\in \mathbb{Z}^d$, there exist finite positive constants $C$ and $s$ depending only on $d,\p,\delta$ and  a nonnegative $\omega$-measurable random variable $\mathcal{T}_\delta^{(1)}(y)$ satisfying
		\begin{equation}\label{equ:random_time}
			\forall T\ge 0, \qquad \P_{\p}\Ll(\mathcal{T}^{(1)}_\delta (y)\ge T\Rr)\le C \exp\left(-\frac{T^s}{C}\right),
		\end{equation}
		such that on the event $\{y\in\cltf\}$, when $t>\mathcal{T}^{(1)}_\delta(y)$, we have
		\begin{equation}\label{equ:W1}
			W_1(p^\omega(t,\cdot+y,y), \bar{p}(t,\cdot))\le K t^\delta,
		\end{equation}
		where $K$ is a finite positive constant depending on $d,\p,\delta$.
	\end{proposition}

	Note that the support of $p^\omega(t,\cdot,y)$ is $\mathscr{C}_\infty\subseteq \mathbb{Z}^d$, while the support of $\bar{p}(t,\cdot)$ is $\mathbb{R}^d$, and they are mutually singular. Therefore, the geometry of $\mathscr{C}_\infty$ must be taken into consideration.  By the triangle inequality and Lemma~\ref{duality}, we can split the $W_1$ distance into two parts as follows:
	\begin{equation}\label{equ:W1_split}
		\begin{aligned}
			&W_1\left(p^\omega(t,\cdot+y,y),\bar{p}(t,\cdot)\right)\\
			&=\sup_{f\in \Lip(\mathbb{R}^d)}\Bigg\vert \sum_{x: x+y\in \mathscr{C}_\infty}f(x)p^\omega(t,x+y,y)-\int_{\mathbb{R}^d}f(x)\bar{p}(t,x)\, \d x\Bigg \vert\\
			&\le I_1+I_2. \\
		\end{aligned}
	\end{equation}
	Here $I_1$ and $I_2$ are defined as follows:
	\begin{equation}\label{eq.defI1I2}
		\begin{split}
			I_1 & := \sup_{f\in \Lip(\mathbb{R}^d)}\Bigg|\sum_{x: x+y\in\mathscr{C}_\infty}f(x)\left(p^\omega(t,x+y,y)-\theta(\mathfrak{p})^{-1}\bar{p}(t,x)\right)\Bigg|,\\
			I_2 & := \sup_{f\in \Lip(\mathbb{R}^d)}\Bigg \vert \sum_{x: x+y\in \mathscr{C}_\infty}f(x)\theta(\mathfrak{p})^{-1}\bar{p}(t,x)-\int_{\mathbb{R}^d}f(x)\bar{p}(t,x)\, \d x\Bigg \vert.
		\end{split}
	\end{equation}
	The term $I_1$ is the approximation of the transition probability supported on $\cltf$, and $I_2$ serves as the approximation of the cluster density.
	For the sake of clarity, we will treat them separately in the following two subsections.
	
	\subsection{Approximation of local CLT}

	As $I_1$ can be considered as the approximation of the transition probability, the following quantitative version of the local CLT for the VSRW established in \cite[Theorem 1.1]{dario2021quantitative} is helpful.
	\begin{lemma}[Quantitative local CLT]\label{quantitative CLT}
		For each exponent ${\delta> 0}$, there exist finite positive constants $C,s$ depending only on the parameters $d,\mathfrak{p}$ and $\delta$, such that for every $y\in\mathbb{Z}^d$, there exists a nonnegative random time $\mathcal{T}_{par,\delta}(y)$ satisfying
		\begin{equation}\label{equ:4}
			\forall T\ge 0,\qquad \P_\p\left(\mathcal{T}_{par,\delta}(y)\ge T\right)\le C\exp\left(-\frac{T^s}{C}\right). 
		\end{equation}
		On the event $\{y\in\cltf\}$, for every $x\in \mathscr{C} _\infty$ and every $t\ge \max\{\mathcal{T}_{par,\delta}(y),\vert x-y\vert\}$, we have
		\begin{align*}
			\vert p^\omega(t,x,y)-\theta(\mathfrak{p})^{-1}\bar{p}(t,x-y)\vert \le C t^{-\frac{d}{2}-(\frac{1}{2}-\delta)} \exp\left(-\frac{\vert x-y\vert^2}{Ct}\right).
		\end{align*}
	\end{lemma}

	Note that the local CLT can only be applied when ${\vert x-y\vert\le t}$, since the jump process exhibits an exponential tail rather than a Gaussian tail in the regime ${\vert x-y\vert> t}$. Thus, we also need the following Carne--Varopoulos bound \cite[Corollaries 11 and 12]{Davies1993Large} to address the case ${\vert x-y\vert> t}$. 
	\begin{lemma}[Carne--Varopoulos bound]\label{Pro: tail}
		There exists $C:=C(d)>0$, such that for every point $x,y \in \mathscr{C}_\infty$, we have
		\[
		p^\omega(t, x, y) \leq \begin{cases}C \exp \left(-\frac{\mathsf{d}(x,y)^{2}}{C t}\right) & \text { if }\mathsf{d}(x,y)\leq t,  \\ C \exp \left(-\frac{\mathsf{d}(x,y)}{C}\left(1+\ln \frac{\mathsf{d}(x,y)}{t}\right)\right) & \text { if }\mathsf{d}(x,y)>t.\end{cases}
		\]
		Here, $\mathsf{d}(x,y)$ denotes the chemical distance on $\cltf$, i.e., minimum length of the open paths connecting $x$ and $y$.
	\end{lemma}
	\begin{remark}
		Since the chemical distance can always be lower bounded by the $\ell_1$ norm, for every $\vert x - y \vert > t$, we have
		\begin{align*}
			\mathsf{d}(x,y) \geq \vert x - y\vert_1 \geq \vert x - y\vert > t.
		\end{align*}
		Using the Carne--Varopoulos bound, we  obtain
		\begin{equation}\label{eq.carne}
			\forall |x-y|>t,\qquad p^\omega(t, x, y) \leq C \exp \left(-\frac{|x-y|}{C}\right).
		\end{equation}
	\end{remark}

	With the two lemmas above, we obtain the following proposition for the term $I_1$.
	\begin{proposition}\label{pro:approximation_CLT}
		For each exponent $\delta>0$, there exists $C:=C(d,\mathfrak{p},\delta)>0$ such that on the event $\{y\in\cltf\}$, when $t>\mathcal{T}_{par,\delta}(y)$, we have
		\begin{equation}\label{equ:approximation_CLT}
			\sup_{f\in \Lip(\mathbb{R}^d)}\Bigg\vert \sum_{x: x + y\in\mathscr{C}_\infty}f(x)\left(p^\omega(t,x+y,y)-\theta(\mathfrak{p})^{-1}\bar{p}(t,x)\right)\Bigg\vert\le Ct^\delta .
		\end{equation}
		Here $\mathcal{T}_{par,\delta}(y)$ comes from Lemma~\ref{quantitative CLT}.
	\end{proposition}
	\begin{proof}
		Without loss of generality, we assume $y=f(0)=0$, which leads to ${\vert f(x)\vert\le \vert x\vert}$. We also use the shorthand $p^\omega(t,x):=p^\omega(t,x,0)$ in the proof. Then $I_1$ can be decomposed as 
		\begin{equation}\label{equ:dec_distance}
			\begin{aligned}
				&\sum_{x\in\mathscr{C}_\infty}f(x)\left(p^\omega(t,x)-\theta(\mathfrak{p})^{-1}\bar{p}(t,x)\right)\\
				&=\sum_{x\in\mathscr{C}_\infty,\vert x\vert\le t}f(x)\left(p^\omega(t,x)-\theta(\mathfrak{p})^{-1}\bar{p}(t,x)\right)\\
				&\qquad+\sum_{x\in\mathscr{C}_\infty, \vert x\vert>t}f(x)\left(p^\omega(t,x)-\theta(\mathfrak{p})^{-1}\bar{p}(t,x)\right).
			\end{aligned}
		\end{equation}  
		As stated, we focus on the long-term behavior when ${t>\mathcal{T}_{par,\delta}(0)}$. Then, Lemma~\ref{quantitative CLT} can be applied to the case $\vert x\vert\le t$: there exists a constant $C_1(d,\mathfrak{p},\delta)>0$ such that
		\begin{align*}
			&\sup_{f\in \Lip(\mathbb{R}^d)}\Bigg\vert \sum_{x\in\mathscr{C}_\infty, \vert x\vert\le t}f(x)(p^\omega(t,x)-\theta(\mathfrak{p})^{-1}\bar{p}(t,x))\Bigg\vert \\
			&\le\sum_{x\in\mathscr{C}_\infty, \vert x\vert\le t}C_1\vert x\vert t^{-\frac{d}{2}-(\frac{1}{2}-\delta)}\exp\left(-\frac{\vert x\vert^2}{C_1t}\right).
		\end{align*}
		We define annuli as follows:
		\begin{align*}
			A_k :=B_{k} \setminus B_{k-1},\quad \forall k \in \N,
		\end{align*}
		where $B_{-1} := \varnothing, B_k:= B(0,2^k)$ is the ball defined in Section~\ref{notation_percolation}. It follows that
		\begin{align*}
			&\sum_{x\in\mathscr{C}_\infty,\vert x\vert\le t}C_1\vert x\vert t^{-\frac{d}{2}-(\frac{1}{2}-\delta)}\exp\left(-\frac{\vert x\vert^2}{C_1t}\right)\\
			&\le C_1\sum_{k=0}^{\infty}\sum_{x\in\mathscr{C}_\infty\cap A_k} 2^k t^{-\frac{d}{2}-(\frac{1}{2}-\delta)}\exp\left(-\frac{(2^{k-1})^2}{C_1t}\right)\\
			&\le 2^{d+1} C_1 t^{\delta}\sum_{k=0}^{\infty}\left(\frac{2^{k-1}}{\sqrt{t}}\right)^{d+1}\exp\left(-\frac{1}{C_1}\left(\frac{2^{k-1}}{\sqrt{t}}\right)^2\right).
		\end{align*}
		Noting that when $x$ is large enough (depending on $C_1$), we have
		\begin{equation}\label{equ:absorb}
			x^{d+1}\exp\left(-\frac{x^2}{C_1}\right)\le \exp\left(-\frac{x^2}{2C_1}\right),
		\end{equation}
		which implies
		\begin{equation}\label{equ:W1_first1}
			\begin{aligned}
				&t^{\delta}\sum_{k=1}^{\infty}\frac{2^{k-1} }{\sqrt{t}}\left(\frac{2^{k-1}}{\sqrt{t}}\right)^{d}\exp\left(-\frac{1}{C_1}\left(\frac{2^{k-1}}{\sqrt{t}}\right)^2\right)\\
				&\le  C_1t^{\delta} \sum_{k=1}^{\infty}\frac{2^{k-1}}{\sqrt{t}}\exp\left(-\frac{1}{2C_1}\left(\frac{2^{k-1}}{\sqrt{t}}\right)^2\right)\\
				&\le C_1t^{\delta} \int_{0}^{\infty}\exp\left(-\frac{x^2}{2C_1}\right)\, \d x < C_1t^{\delta}.
			\end{aligned}
		\end{equation}
		In the last inequality, since the exponential function is decreasing, the discrete sum can be bounded by integration. This yields an upper bound for the case $\vert x\vert\le t$ in \eqref{equ:dec_distance}.
		
		\medskip
		For the case $|x-y|>t$, by \eqref{eq.carne} of Lemma~\ref{Pro: tail}, there exists a constant $C_2(d,\mathfrak{p})>0$ such that
		\begin{align*}
			&\sup_{f\in \Lip(\mathbb{R}^d)}\Bigg\vert \sum_{x\in\mathscr{C}_\infty, |x|>t}f(x)(p^\omega(t,x)-\theta(\mathfrak{p})^{-1}\bar{p}(t,x))\Bigg\vert \\
			&\le \sum_{k=\lfloor \log_2 t\rfloor }^{\infty}\sum_{x\in\mathscr{C}_\infty\cap A_k}\vert x\vert(p^\omega(t,x)+\theta(\mathfrak{p})^{-1}\bar{p}(t,x))\\
			&\le C_2\sum_{k=\lfloor \log_2 t\rfloor}^{\infty}(2^k)^{d+1}\exp\left(-\frac{2^{k-1}}{C_2}\right)+C_2t^{-\frac{d}{2}}\sum_{k=\lfloor \log_2 t\rfloor}^{\infty}(2^k)^{d+1}\exp\left(-\frac{(2^{k-1})^2}{2\bar{\sigma}^2t}\right).
		\end{align*}
		For the two terms in the last equation mentioned above, the same argument as in \eqref{equ:absorb} and \eqref{equ:W1_first1} can be applied. Therefore, when $k$ is sufficiently large, it follows that 
		\begin{align}
			\sum_{k=\lfloor \log_2 t\rfloor }^{\infty}(2^k)^{d+1}\exp\left(-\frac{2^{k-1}}{C_2}\right) &\le C_2\exp\left(-\frac{t}{2C_2}\right), \\
			t^{-\frac{d}{2}}\sum_{k=\lfloor \log_2 t\rfloor}^{\infty}(2^k)^{d+1}\exp\left(-\frac{(2^{k-1})^2}{2\bar{\sigma}^2t}\right)&\le C_2t^{-\frac{1}{2}}\exp\left(-\frac{t^2}{4\bar{\sigma}^2}\right)\label{eq.W1_first3}.
		\end{align}
		
		Combining \eqref{equ:W1_first1}-\eqref{eq.W1_first3}, we conclude that when $t>\mathcal{T}_{par,\delta}(0)$,  there exists a finite constant $C(d,\mathfrak{p},\delta)>0$ such that
		$$\sup_{f\in \Lip(\mathbb{R}^d)}\Bigg\vert \sum_{x\in\mathscr{C}_\infty}f(x)(p^\omega(t,x)-\theta(\mathfrak{p})^{-1}\bar{p}(t,x))\Bigg\vert\le Ct^\delta,$$
		which completes the proof.
	\end{proof}

	\subsection{Approximation of cluster density}
	In this subsection, we establish an upper bound for $I_2$ defined in \eqref{eq.defI1I2} in the following proposition.
	
	\begin{proposition}\label{pro:approximation_cluster}
		There exist finite positive constants $C_1$ and $s_1$ that depend on $d,\p$ such that, for every $t>0$, the following holds:
		\begin{multline}\label{equ:Orlicz}
			\sup_{f\in \Lip(\mathbb{R}^d)} \Bigg\vert \sum_{x: x + y\in\mathscr{C}_\infty}f(x)\theta(\mathfrak{p})^{-1}\bar{p}(t,x)-\int_{\mathbb{R}^d}f(x)\bar{p}(t,x)\, \d x\Bigg\vert \\
			\le 
			\left\{ 		
			\begin{array}{ll}
				\mathcal{O}_{s_1}\left(C_1\lfloor\frac{1}{2}\log_3t \rfloor\right),  & d=2,  \\ 
				\mathcal{O}_{s_1}(C_1), & d\ge 3.
			\end{array}\right. 
		\end{multline}

		Moreover, for each $\delta >0$, there exist finite positive constants $C_2$ and $s_2$ that depend on $d,\p,\delta$ and a nonnegative $\omega$-measurable random variable $\mathcal{T}_{dense,\delta}(y)$ satisfying
		\begin{equation}\label{equ:random_time2}
			\forall T\ge 0, \qquad \P_{\p}\Ll(\mathcal{T}_{dense,\delta}(y)\ge T\Rr)\le C_2 \exp\left(-\frac{T^{s_2}}{C_2}\right),
		\end{equation}
		such that on the event $\{y\in\cltf\}$, for every $t>\mathcal{T}_{dense,\delta}(y)$, we have 
		\begin{equation}\label{equ:scale}
			\sup_{f\in \Lip(\mathbb{R}^d)}\Bigg\vert \sum_{x: x + y\in\mathscr{C}_\infty}f(x)\theta(\mathfrak{p})^{-1}\bar{p}(t,x)-\int_{\mathbb{R}^d}f(x)\bar{p}(t,x)\, \d x\Bigg\vert\le C_2t^\delta.
		\end{equation}
	\end{proposition}
	
	The main ingredient to prove Proposition~\ref{pro:approximation_cluster} is the convergence of density that
	\begin{align*} %\label{eq.densityErgodic}
		{\frac{\vert \cltf \cap B(0,r)\vert}{\vert B(0,r)\vert} \xrightarrow{r \to \infty} \theta(\p)},
	\end{align*}
	which is derived from ergodic theory. To obtain a quantitative result, the convergence rate established in \cite[Proposition~14]{dario2021quantitative} is crucial. We begin by introducing some necessary notation and stating this result.
	A triadic cube in $\Zd$ is of type
	\begin{align*}
		\square_m(z):=\mathbb{Z}^d\cap\left(z+\Big(-\frac{3^m}{2},\frac{3^m}{2}\Big)^d\right),\qquad z\in 3^m\mathbb{Z}^d,\,m\in\mathbb{N}.
	\end{align*}
	For every bounded set $U \subset \Zd$ and $p\in[1,\infty)$, we define the $L^p(U)$-norm and the normalized $\underline{L}^p(U)$-norm as 
	\begin{align*}
		\|u\|_{L^p(U)} :=\left(\sum_{x \in U}|u|^p(x)\right)^{\frac{1}{p}}, \qquad
		\|u\|_{\underline{L}^p(U)} :=\left(\frac{1}{|U|}\sum_{x \in U} |u|^p(x)\right)^{\frac{1}{p}}.
	\end{align*}
	Let $\D u$ denote the finite difference vector of $u$, defined by
	\begin{align*}
		\D u (x) := \big(u(x+e_1)- u(x), u(x+e_2)- u(x), \cdots u(x+e_d)- u(x)\big)^\top,
	\end{align*}
	where $\{e_i\}_{i=1}^n$ is the canonical basis of $\R^d$.
	The normalized discrete Sobolev norm $\underline{H}^1(U)$ is defined by 
	\begin{align}\label{eq.defSobolev}
		\|u\|_{\underline{H}^1(U)} :=|U|^{-\frac{1}{d}}\|u\|_{\underline{L}^2(U)}+\|\D u\|_{\underline{L}^2(U)},
	\end{align}
	and we define its dual norm as 
	\begin{align}\label{eq.defH-1}
		\|u\|_{\underline{H}^{-1}(U)} :=\sup_{\|v\|_{\underline{H}^1(U)}\le 1}\frac{1}{|U|}\sum_{x \in U} u(x)v(x).
	\end{align}
	Since $\underline{H}^{-1}$-norm effectively captures the spatial cancellation, the following lemma provides a quantitative measure of the concentration rate for the cluster density. For details, see \cite[Step 2 of the proof for Proposition 14]{dario2021quantitative}.
	\begin{lemma}\label{Lem:H}
		For each $z\in\mathbb{Z}^d$ and $m\in\mathbb{N}$, the following holds:
		\begin{align}
			\|\mathbf{1}_{\cltf}-\theta(\p)\|_{\underline{H}^{-1}\left(\square_m(z)\right)}\le 
			\left\{\begin{array}{ll}
				\mathcal{O}_s(C m),  & d=2,  \\ 
				\mathcal{O}_s(C), & d\ge 3,
			\end{array}\right. 
		\end{align}
		where $C$ and $s$ are finite positive constants depending on $d,\p$.
	\end{lemma}

	\begin{proof}[Proof of Proposition~\ref{pro:approximation_cluster}]
		Without loss of generality, we assume $y=f(0)=0$, which leads to $|f(x)|\le |x|$.  For the term $I_2$ in \eqref{eq.defI1I2}, the triangle inequality implies 
		
		\begin{multline}\label{eq.densityClusterDecom}
			\Bigg\vert \sum_{x\in \mathscr{C}_\infty}f(x)\theta(\mathfrak{p})^{-1}\bar{p}(t,x)-\int_{\mathbb{R}^d}f(x)\bar{p}(t,x)\, \d x\Bigg\vert\\
			\le\theta(\p)^{-1}\Bigg\vert\sum_{ x \in \mathbb{Z}^d}f(x)\bar{p}(t,x)\left(\mathbf{1}_{\{x\in \cltf\}}-\theta(\p)\right)\Bigg\vert
			+\Bigg\vert \sum_{ x \in \mathbb{Z}^d}f(x)\bar{p}(t,x)-\int_{\mathbb{R}^d}f(x)\bar{p}(t,x)\,\d x\Bigg\vert.
		\end{multline}

		We first address the second term by decomposing $\mathbb{R}^d$ into a union of unit cubes $\square_0(x)$ centered at $x\in\mathbb{Z}^d$. 
		\begin{equation}\label{equ:expectation_B}
			\begin{aligned}
				&\Bigg \vert \sum_{x\in\mathbb{Z}^d}f(x)\bar{p}(t,x)-\int_{\mathbb{R}^d}f(x)\bar{p}(t,x)\,\d x\Bigg\vert\\
				\le& \sum_{x\in\mathbb{Z}^d} \Big\vert f(x)\bar{p}(t,x) -\int_{\square_0(x)}f(u)\bar{p}(t,u)\,\d u\Big\vert\\
				\le &\sum_{x\in\mathbb{Z}^d}\vert f(x)\vert \int_{\square_0(x)}\vert\bar{p}(t,x)-\bar{p}(t,u)\vert \,\d u+\sum_{x\in\mathbb{Z}^d}\int_{\square_0(x)}\vert f(x)-f(u)\vert\bar{p}(t,u)\,\d u.
			\end{aligned}
		\end{equation}

		For the first term on the {\rhs} of \eqref{equ:expectation_B}, the mean-value theorem and $1$-Lipschitz property of $f$ imply
		\begin{equation*}
			\sum_{x\in\mathbb{Z}^d}\vert f(x)\vert \int_{\square_0(x)}\vert\bar{p}(t,x)-\bar{p}(t,u)\vert \,\d u\le \sum_{x\in\mathbb{Z}^d}\vert x\vert\int_{\square_0(x)}\vert  (x-u) \cdot \nabla\bar{p}(t,\xi(u))  \vert \,\d u,
		\end{equation*}
		where $\xi(u)$ denotes the intermediate point. Since $x,u$ and $\xi(u)$ lie in the same unit cube, we have
		\begin{equation*}
			\sum_{x\in\mathbb{Z}^d}\vert x\vert\int_{\square_0(x)}\vert  (x-u) \cdot \nabla\bar{p}(t,\xi(u))  \vert \,\d u\le  C \sum_{x\in\mathbb{Z}^d} \vert x\nabla \bar{p}(t,x)\vert \le C,
		\end{equation*}
		where the last inequality follows from a calculation similar to that in the proof of Proposition~\ref{pro:approximation_CLT}. Thus,
		\begin{equation}\label{equ:expectation_B1}
			\sum_{x\in\mathbb{Z}^d}\vert f(x)\vert \int_{\square_0(x)}\vert\bar{p}(t,x)-\bar{p}(t,u)\vert \,\d u\le C.
		\end{equation}
		
		For the second term on the {\rhs} of \eqref{equ:expectation_B}, since $f$ is $1$-Lipschitz and $x,u$ belong to the same unit cube, we can obtain
		\begin{equation}\label{equ:expectation_B2}
			\begin{aligned}
				\sum_{x\in\mathbb{Z}^d}\int_{\square_0(x)}\vert f(x)-f(u)\vert\bar{p}(t,u)\,\d u \le  C\sum_{x\in\mathbb{Z}^d}\int_{\square_0(x)}\bar{p}(t,u)\,\d u = C.
			\end{aligned}
		\end{equation}
		Therefore, \eqref{equ:expectation_B1} and \eqref{equ:expectation_B2} yield a constant bound for the second term on the {\rhs} of \eqref{eq.densityClusterDecom}
		\begin{align}\label{equ:expectation_B_T2}
			\Bigg\vert \sum_{ x \in \mathbb{Z}^d}f(x)\bar{p}(t,x)-\int_{\mathbb{R}^d}f(x)\bar{p}(t,x)\,\d x\Bigg\vert \leq C.
		\end{align}
		
		\smallskip
		
		Now we turn to the first term on the {\rhs} of \eqref{eq.densityClusterDecom} and decompose it as the sum over the triadic cubes $\square_m(z)$ with center at $z\in 3^m\mathbb{Z}^d$
		\begin{equation}\label{eq.densityClusterKey}
			\begin{aligned}
				&\sup_{f\in \Lip(\mathbb{R}^d)} \Bigg\vert\sum_{x\in\mathbb{Z}^d}f(x)\bar{p}(t,x)\left(\mathbf{1}_{\{x\in \cltf\}}-\theta(\p)\right)\Bigg\vert\\
				&\le\sup_{f\in \Lip(\mathbb{R}^d)}\sum_{z\in 3^m\mathbb{Z}^d} \Bigg\vert\sum_{x\in\square_m(z)}f(x)\bar{p}(t,x)\left(\mathbf{1}_{\{x\in \cltf\}}-\theta(\p)\right)\Bigg\vert\\
				&\le 3^{dm} \sum_{z\in 3^m\mathbb{Z}^d} \|\mathbf{1}_{\cltf}-\theta(\p)\|_{\underline{H}^{-1} \left(\square_m(z)\right)} \Big(\sup_{f\in \Lip(\mathbb{R}^d)} \|f\bar{p}(t,\cdot)\|_{\underline{H}^1\left(\square_m(z)\right)}\Big),
			\end{aligned}
		\end{equation}
		where we use the definition \eqref{eq.defH-1} in the last inequality. It remains to calculate the $\underline{H}^1$-norm of $f \bar{p}(t,\cdot)$. It follows from \eqref{eq.defSobolev} and $1$-Lipschitz property of $f$ that 
		\begin{equation}\label{eq.fpH1}
			\begin{aligned}
				&\|f \bar{p}(t,\cdot)\|_{\underline{H}^1\left(\square_m(z)\right)}\\
				&=3^{-\frac{md}{2}}\left(3^{-m}\|f \bar{p}(t,\cdot)\|_{L^2(\square_m(z))}+\|\mathcal{D} (f \bar{p}(t,\cdot))\|_{L^2(\square_m(z))}\right)\\
				&\le 3^{-\frac{md}{2}}\left(3^{-m}\|f \bar{p}(t,\cdot)\|_{L^2(\square_m(z))}+\|\bar{p}(t,\cdot)\|_{L^2(\square_m(z))}+\|f\nabla \bar{p}(t,\cdot)\|_{L^2(\square_m(z))}\right).
			\end{aligned}
		\end{equation}
		
		To control the three terms in the last line of \eqref{eq.fpH1}, the following elementary estimate is useful for every $g\in L^\infty(\square_m(z))$
		\begin{align*}
			\| g\bar{p}(t,\cdot) \|_{L^2(\square_m(z))}\leq \| g^2\bar{p}(t,\cdot) \|^{\frac{1}{2}}_{L^\infty(\square_m(z))} \| \bar{p}(t,\cdot) \|^{\frac{1}{2}}_{L^1(\square_m(z))}<\| g^2\bar{p}(t,\cdot) \|^{\frac{1}{2}}_{L^\infty(\square_m(z))},
		\end{align*}
		where the last inequality follows from $\| \bar{p}(t,\cdot) \|_{L^1(\square_m(z))} < 1$ for every $z\in 3^m\mathbb{Z}^d$.
		
		In particular, together with $\vert f(x) \vert \leq \vert x \vert$ and $\vert \nabla \bar{p}(t,x) \vert \leq C t^{-1}\vert x\vert\bar{p}(t,x)$, we can obtain
		\begin{align}
			3^{-m}\|f \bar{p}(t,\cdot)\|_{L^2(\square_m(z))}&\le 3^{-m}\sqrt{\max_{x\in\square_m(z)}|x|^2\bar{p}(t,x)}\label{eq.L1'},\\
			\|\bar{p}(t,x)\|_{L^2(\square_m(z))}&\le\sqrt{\max_{x\in\square_m(z)}\bar{p}(t,x)}\label{eq.L2'},\\
			\|f\nabla \bar{p}(t,\cdot)\|_{L^2(\square_m(z))}&\le \sqrt{\max_{x\in\square_m(z)}|x|^4 t^{-2}\bar{p}(t,x)}.\label{eq.L3'}
		\end{align}

		To control the maximum terms in the three displays above, we use the following observation
		\begin{equation}\label{eq.max_compare}
			\forall 3^m\le t^{\frac{1}{2}}, z\in3^m\mathbb{Z}^d,\qquad \max_{x\in\square_m(z)}\bar{p}(t,x)\le C\bar{p}(t,z).
		\end{equation}
		Inserting \eqref{eq.max_compare} back to \eqref{eq.L1'}-\eqref{eq.L3'}, we have
		\begin{align}
			3^{-m}\|f \bar{p}(t,\cdot)\|_{L^2(\square_m(z))}&\le C(3^{-m} \vert z \vert + 1)t^{-\frac{d}{4}}\exp\left(-\frac{|z|^2}{4\bar{\sigma}^2 t}\right)\label{eq.L1},\\
			\|\bar{p}(t,x)\|_{L^2(\square_m(z))}&\le Ct^{-\frac{d}{4}}\exp\left(-\frac{|z|^2}{4\bar{\sigma}^2 t}\right)\label{eq.L2},\\
			\|f\nabla \bar{p}(t,\cdot)\|_{L^2(\square_m(z))}&\le C(\vert z \vert + 3^m)^{2}t^{-\frac{d}{4}-1}\exp\left(-\frac{|z|^2}{4\bar{\sigma}^2 t}\right).\label{eq.L3}
		\end{align}
		
		The estimates in \eqref{eq.L1}-\eqref{eq.L3} provide a uniform upper bound for $f \in \Lip(\R^d)$. 
		We substitute these estimates back into \eqref{eq.fpH1} and \eqref{eq.densityClusterKey}, and apply the $\underline{H}^{-1}$-estimate in Lemma~\ref{Lem:H} along with \eqref{equ:Orlicz_sum} to conclude that
		\begin{align*}
			\sup_{f\in \Lip(\mathbb{R}^d)} \Bigg\vert\sum_{x\in\mathbb{Z}^d}f(x)\bar{p}(t,x)\left(\mathbf{1}_{\{x\in \cltf\}}-\theta(\p)\right)\Bigg\vert \leq \left\{ 		
			\begin{array}{ll}
				\mathcal{O}_{s}\left(C A m\right),  & d=2,  \\ 
				\mathcal{O}_{s}(C A), & d\ge 3,
			\end{array}\right. 
		\end{align*}
		where $A$ is the following constant 
		\begin{align*}
			A = 3^{dm} \sum_{z\in 3^m\mathbb{Z}^d} 3^{-\frac{dm}{2}}t^{-\frac{d}{4}}\exp\left(-\frac{|z|^2} {4\bar{\sigma}^2 t}\right) \Ll(1 + ( \vert z\vert + 3^m) 3^{-m} + ( \vert z \vert + 3^m)^{2} t^{-1} \Rr).
		\end{align*}
		Let $m \in \N_+$ be chosen such that $3^m \simeq t^{\frac{1}{2}}$, i.e., $m=\lfloor\frac{1}{2}\log_3 t\rfloor$. Then, we obtain
		\begin{align*}
			A \leq  \sum_{x\in \mathbb{Z}^d} t^{-\frac{d}{2}}\exp\left(-\frac{|x|^2} {4\bar{\sigma}^2 t}\right) \Ll(1 + \frac{\vert x \vert}{\sqrt{t}} + \frac{\vert x \vert^2}{t} \Rr) < C,
		\end{align*}
		where $C$ is a constant independent of $m$ and $t$. This implies 
		\begin{multline}\label{equ:expectation_B_T1}
			\sup_{f\in \Lip(\mathbb{R}^d)}\left\vert\int_{\mathbb{Z}^d}f(x)\bar{p}(t,x)\left(\mathbf{1}_{\{x\in \cltf\}}-\theta(\p)\right)\,\d x\right\vert \\ 
			\le \left\{ 		
			\begin{array}{ll}
				\mathcal{O}_{s}\left(C_1\lfloor\frac{1}{2}\log_3t \rfloor\right),  & d=2,  \\ 
				\mathcal{O}_{s}(C_1), & d\ge 3.
			\end{array}\right. 
		\end{multline}
		Combining \eqref{equ:expectation_B_T1} and \eqref{equ:expectation_B_T2}, we derive \eqref{equ:Orlicz}.
		
		\smallskip
		
		Finally, for every $\delta>0$, we define the minimal scale
		\begin{align*}
			\mathcal{T}_{dense,\delta}(0):=\sup \Bigg\{t>0: t^{-\delta} \Bigg|\sum_{x: x\in\mathscr{C}_\infty}f(x) \bar{p}(t, x)-\int_{\mathbb{R}^d}f(x)\bar{p}(t,x)\,\d x\Bigg| \geq 1\Bigg\}.
		\end{align*}
		By Lemma~\ref{lem:scale}, there exist two positive constants $s_2$ and $C_2$ that depend on $d,\p,\delta$ such that \eqref{equ:random_time2} holds as desired.
	\end{proof}
	
	We are now ready to prove Proposition \ref{W_1}.
	
	\begin{proof}[Proof of Proposition 3.5]
		For each $\delta>0$, we define 
		\begin{align}\label{eq.defT1}
			\mathcal{T}^{(1)}_\delta(y):=\max\left\{\mathcal{T}_{par,\delta}(y),\mathcal{T}_{dense,\delta}(y)\right\}.
		\end{align}
		It follows from \eqref{equ:4} and \eqref{equ:random_time2} that \eqref{equ:random_time} holds.
		For every $t>\mathcal{T}^{(1)}_\delta (y)$, by \eqref{equ:W1_split},\eqref{equ:approximation_CLT} and \eqref{equ:scale}, we obtain \eqref{equ:W1} as desired.
	\end{proof}

	\section{Coupling of processes}\label{sec.CouplingProcess}
	This section is devoted to proving Theorem~\ref{The:main_result}, which consists of two subsections. In Subsection~\ref{Finite}, we demonstrate that for every sufficiently large time $T$, we can construct $(S_t,\B_t)_{t\in [0,T]}$ satisfying the estimate \eqref{eq.main} in the same probability space. 
	In Subsection~\ref{infinite}, we extend the coupled process $(S_t,\bar{B}_t)_{t\in[0,T]}$ to $(S_t,\bar{B}_t)_{t\ge 0}$ and conclude Theorem~\ref{The:main_result}.
	
	\subsection{Coupling in finite horizon}\label{Finite}
	We begin by constructing a coupled process in a finite horizon $[0,T]$. To this end, we define another minimal scale for every $\delta>0$ and $y\in\mathbb{Z}^d$
	\begin{equation}\label{equ:scale3}
		\mathcal{T}^{(3)}_\delta(y):\Omega\to [0,\infty).
	\end{equation}
	It satisfies the stretched exponential tail estimate
	\begin{align}\label{eq.TailT3}
		\forall T>0,\quad \P_\p\left(\mathcal{T}^{(3)}_\delta(y)\geq T\right)\le C\exp\left(-\frac{T^s}{C}\right),
	\end{align}
	where $C,s$ are all finite positive constants depending on $\p,d,\delta$.

	\begin{proposition}\label{The:main_result_finite}
		For almost every configuration $\omega\in \Omega$, we can construct a version of $(S_t)_{t \in[0,T]}$
		and $(\B_t)_{t \in [0,T]}$ in the same probability space, both starting from $y \in \cltf$, such that for every $\delta >0$ and for all $T > \T^{(3)}_\delta(y)$, we have
		\begin{align}\label{eq.main_restate}
			\E^\omega\Big[\sup_{t\in [0,T]} \vert S_t - \B_t\vert \Big] \leq K T^{\frac{1}{3}+\delta},
		\end{align}
		where $K$ is a finite positive constant depending on $\p,d,\delta$.
	\end{proposition}

	The proof of Proposition \ref{The:main_result_finite} consists of three ingredients: 
	\begin{itemize}
		\item Constructing the coupled process  $(S_t,\bar{B}_t)_{t\in[0,T]}$;
		\item Verifying the consistency of the coupling;
		\item Estimating the distance.
	\end{itemize}
	The first two ingredients are provided in Subsection~\ref{constructing the process}, and the last is provided in Subsection~\ref{bounding the typical distance}.

	\subsubsection{Construction of the coupled process}\label{constructing the process}
	\begin{definition}[Coupling in finite horizon]\label{def.CFH}
		We denote by $\CFH$$(T,n,y)$ the coupling of $(S_{t},\B_{t})_{0\leq t \leq T}$ constructed in the following way. We define that 
		\begin{align}\label{eq.defInterval_again}
			\Delta T = T/n, \quad t_k = k \Delta T, \quad I_k = (t_k, t_{k+1}),	\qquad k\in\{0,1,\dots, n-1\}.
		\end{align}
		\begin{enumerate}[label=(\Alph*)]
			\item The processes start from $S_0 = \B_0 = y$.
			\item Conditioned on $\big(S_{t_i},\B_{t_i}\big)_{0\le i\le k}$, we sample the increment $\big(\Delta S_{t_k}, \Delta\B_{t_k}\big)$ 
			\begin{align*} %\label{eq.defDeltaS_Coupling}
				\big(\Delta S_{t_k}, \Delta\B_{t_k}\big)  \in \varPi\big(p^\omega(\Delta T,S_{t_k} +\cdot,S_{t_k}), \bar{p}(\Delta T,\cdot) \big),
			\end{align*}
			 such that
			\begin{align*} %\label{eq.defDeltaS_Total}
				\E^\omega \Ll[\Ll\vert \Delta S_{t_k}-\Delta\B_{t_k}\Rr\vert \Big\vert \big(S_{t_i},\B_{t_i}\big)_{0\le i\le k} \Rr]=W_1\big(p^\omega(\Delta T,S_{t_k}+\cdot,S_{t_k}), \bar{p}(\Delta T,\cdot) \big).
			\end{align*}
			Then we  define 
			\begin{align*} %\label{eq.prolong}
				\big(S_{t_{k+1}},\B_{t_{k+1}}\big) := \big(S_{t_{k}},\B_{t_{k}}\big) + \big(\Delta S_{t_k}, \Delta\B_{t_k}\big).
			\end{align*}
			\item Given the sequence $\big(S_{t_k},\B_{t_k}\big)_{0\le k\le n}$, we sample the trajectories as follows:
			\begin{itemize}[label=---]
				\item $(S_t)_{t \in [0,T]}$ is sampled as the VSRW defined in \eqref{eq.VSRW} conditioned on the values $(S_{t_k})_{0\le k\le n}$ at the endpoints $(t_k)_{0\le k\le n}$.
				\item $(\B_t)_{t \in [0,T]}$ is sampled as Brownian motion with diffusion constant $\bar{\sigma}^2$ conditioned on the values $(\B_{t_k})_{0\le k\le n}$ at the endpoints $(t_k)_{0\le k\le n}$.
			\end{itemize}
		\end{enumerate}
	\end{definition}

	We refer to Lemma~\ref{lem.existence} for the existence of $\CFH$, and the following proposition justifies its consistency.
	
	\begin{proposition}\label{Pro:distribution}
		The process $\CFH$$(T,n,y)$ is a coupling between the VSRW and Brownian motion with diffusive constant $\bar{\sigma}$ on the interval $[0,T]$.
	\end{proposition}
	\begin{proof}
		Since the trajectory of $(S_t)_{t\in[0,T]}$ is c\`adl\`ag and $(\B_t)_{t\in[0,T]}$ is continuous, it suffices to verify the consistency of finite-dimensional distribution. We first deal with $(S_t)_{t\in[0,T]}$. Denoting $y_0 = y$, the chain rule of conditional probability implies 
		\begin{align*}
			\P^\omega\left(S_{t_k}=y_k,k=1,\dots,n\right)
			%=&\prod_{i=1}^n\P^\omega\left(S_{t_i}=y_i\vert S_{t_k}=y_k, k=0,\dots,i-1\right)\\
			=\prod_{i=1}^n\P^\omega\Ll( \Delta S_{t_{i-1}} = y_i-y_{i-1} \vert S_{t_0} = y_0, \cdots ,S_{t_{i-1}} = y_{i-1}\Rr).
		\end{align*}
		We then use (B) of Definition~\ref{def.CFH} and the property of conditional expectation
		\begin{align*}
			\P^\omega\Ll( \Delta S_{t_{i-1}} = y_i-y_{i-1} \vert S_{t_0} = y_0, \cdots ,S_{t_{i-1}} = y_{i-1}\Rr)
			%&= \P^\omega\Ll( \P^\omega \big(\Delta S_{t_{i-1}} = y_i-y_{i-1}\vert (S_{t_j},\B_{t_j})_{0\le j\le i-1}\big)\vert S_{t_0} = y_0, \cdots ,S_{t_{i-1}} = y_{i-1}\Rr)\\
			%= \P^\omega\Ll( p^\omega(\Delta T,S_{t_i-1} +y_i-y_{i-1},S_{t_i-1}) \vert S_{t_0} = y_0, \cdots ,S_{t_{i-1}} = y_{i-1}\Rr)\\
			= p^\omega(\Delta T,y_i,y_{i-1}).
		\end{align*}
		This leads to the following identity
		\begin{equation*}
			\P^\omega\left(S_{t_k}=y_k,k=1,\dots,n\right)
			=\prod_{i=0}^{n-1}p^\omega(\Delta T, y_i, y_{i-1}).
		\end{equation*}
		
		Consequently, we obtain that $(S_{t_k})_{0\le k\le n}$ is distributed identically to the VSRW at the endpoints $(t_k)_{0\le k\le n}$. Then,  the consistency of $(S_t)_{t\in[0,T]}$ is ensured by the natural extension in (C).

		The proof for $(\B_t)_{t\in[0,T]}$ is similar, and the details are omitted.
	\end{proof}

	\subsubsection{Distance between $(S_t)_{t\in[0,T]}$ and $(\B_t)_{t\in[0,T]}$}\label{bounding the typical distance}
	The objective of this subsection is to prove \eqref{eq.main_restate} for the coupling in Definition~\ref{def.CFH}. 
	
	\begin{proposition}\label{prop.CFH_bound}
		In the setting of Proposition~\ref{The:main_result_finite}, $\CFH$$(T, n, y)$ with ${n = \lfloor T^{\frac{1}{3}} \rfloor}$ satisfies the estimate \eqref{eq.main_restate}.
	\end{proposition}

	The choice of $n = \lfloor T^{\frac{1}{3}} \rfloor$ is motivated by the heuristic analysis in Section~\ref{subsec.ingredient}. To prove Proposition \ref{prop.CFH_bound}, the following lemma from \cite[Proposition 4.7]{ABDH} is useful to control the maximum displacement of VSRW in every interval.
	
	We define the function
	\[
	\Psi(R,T)= \begin{cases}\exp\left(-\frac{R^2}{T}\right) & \text { if } R \leq e T,\\ \exp\left(-R\log\left(\frac{R}{T}\right)\right) & \text { if } R > e T.\end{cases}
	\]
	\begin{lemma}\label{Pro:Barlow}
		For every $y\in\mathbb{Z}^d$, there exist finite positive constants $C$ and $s$ depending on $d,\p$ and a nonnegative $\omega$-measurable random variable $\mathcal{R}(y)$ satisfying 
		\begin{equation}\label{equ:R_tail}
			\forall r>0,\quad \P_\p(\mathcal{R}(y)\ge r)\le C\exp\left(-\frac{r^s}{C}\right),
		\end{equation} 
		such that for almost every $\omega\in\Omega$ and every $y \in \cltf$, when $R>\mathcal{R}(y)$, we have
		\begin{equation}\label{equ:maxS_tail}
			\forall T>0,\quad P^\omega_y\Big(\sup_{t\in[0,T]}\vert S_t-y\vert>R\Big)\le C_1\Psi\left(C_2R,T\right),
		\end{equation}
		where $C_1,C_2$ are finite positive constants depending on $d,\p$.
		
	\end{lemma}
	
	The following corollary is derived from the above lemma.
	\begin{corollary}\label{cor.S_max}
		There exists $C:=C(d,\p)>0$, such that for almost every $\omega\in\Omega$ and every $y\in\cltf$, for all
		$T \geq \mathcal{R}(y)^2$, we have
		\begin{align}\label{eq.maxSt}
			\sup_{t\in[0,T]}\vert S_t-y\vert \leq \oO^\omega_1(C \sqrt{T}),
		\end{align}
		where $\oO^\omega_1(\cdot)$ is associated with $P^\omega_y$.
	\end{corollary}
	\begin{proof}
		By \eqref{equ:tail_1}, it suffices to derive an estimate for the tail probability. Since ${\mathcal{R}(y)\ge 1}$ by convention, we always have $\mathcal{R}^2(y)\ge\mathcal{R}(y)$ and $T\ge 1$. Then, the proof is divided into three cases.

		\emph{Case $R\ge eT$}: By $T\ge\mathcal{R}(y)^2\ge \mathcal{R}(y)$,  Lemma \ref{Pro:Barlow} implies 
		\begin{equation}\label{equ:maxSt_tail1}
			P^\omega_y\Big(\sup_{t\in[0,T]}\vert S_t-y\vert>R\Big)\le C_1\exp\left(-C_2R\right)\le C_1\exp\left(-\frac{C_2R}{\sqrt{T}}\right).
		\end{equation}
		
		\emph{Case $\sqrt{T}\le R< eT$}: By $T\ge\sqrt{T}\ge \mathcal{R}(y)$, Lemma \ref{Pro:Barlow} gives us
		\begin{equation}\label{equ:maxSt_tail2}
			P^\omega_y\Big(\sup_{t\in[0,T]}\vert S_t-y\vert>R\Big)\\
			\le C_1\exp\left(-\frac{C^2_2R^2}{T}\right)\le C_1\exp\left(-\frac{C^2_2 R}{\sqrt{T}}\right).
		\end{equation}
		
		Combining \eqref{equ:maxSt_tail1} and \eqref{equ:maxSt_tail2}, there exists a constant $C(d,\p)>0$ such that
		\begin{equation*}
			\forall R\ge\sqrt{T},\qquad P^\omega_y\Big(\sup_{t\in[0,T]}\vert S_t-y\vert>R\Big)\le C \exp\left(-\frac{R}{C\sqrt{T}}\right).
		\end{equation*}
		It remains to show that the above inequality holds when $R\in [0,\sqrt{T})$.
		
		\emph{Case $0\le R<\sqrt{T}$}: for sufficiently large $C$, we obtain a naive bound
		\begin{equation}
			P^\omega_y\Big(\sup_{t\in[0,T]}\vert S_t-y\vert>R\Big) \leq 1\le C \exp\left(-\frac{1}{C}\right)\le C \exp\left(-\frac{R}{C\sqrt{T}}\right),
		\end{equation}
		which leads to the desired result.
	\end{proof}

	\begin{proof}[Proof of Proposition~\ref{prop.CFH_bound}]
		We recall the definition in \eqref{eq.defInterval_again} and maintain that $n = \lfloor T^{\frac{1}{3}} \rfloor$ and $\Delta T \simeq T^{\frac{2}{3}}$ throughout the proof.  The proof can be divided into three steps.
		
		\textit{Step~1: set up the the minimal scales.}
		Define a good event that captures the typical case
		\begin{align}\label{eq.defGoodEvent}
			D:= \bigcap_{k=0}^{n-1}\big\{S_{t_k}\in B(y,T)\big\}. 
		\end{align}
		
		To facilitate subsequent calculations, we decompose the distance between the processes as follows 
		\begin{equation}
			\begin{aligned}\label{equ:split}
				&\E^\omega\Big[\sup_{t\in[0,T]}\vert S_t-\bar{B}_t\vert\Big] \\
				&\le\E^\omega\Big[\sup_{t\in [0,T]}\vert S_t-\bar{B}_t\vert\mathbf{1}_D\Big]+\E^\omega\Big[\sup_{t\in [0,T]}\vert S_t-\B_t\vert\mathbf{1}_{D^c}\Big],
			\end{aligned}
		\end{equation}
		where the first term on the {\rhs} represents the typical distance, while the second term accounts for the atypical distance.
		
		We will use Proposition \ref{W_1} and Corollary \ref{cor.S_max} to control the typical distance. It is worth noting that the above two results hold true only for sufficiently large $t$.
		Therefore, we need to introduce the following minimal scales.	
		
		To apply Proposition \ref{W_1}, we need to define the minimal scale
		\begin{align}\label{eq.defT2}
			\mathcal{T}_{\delta}^{(2)}(y) :=\begin{cases} \sup\left\{3^{n}: \max_{x\in B(y,3^n)}\mathcal{T}^{(1)}_\delta (x)\ge 3^{\frac{2(n-1)}{3}} \right\}& \text {on~}\{y\in\cltf\},\\ 0 & \text {on~}\{y\notin\cltf\}.\end{cases}
		\end{align}
		From \eqref{equ:random_time}, it follows that $\mathcal{T}_{\delta}^{(1)}(x)\le \mathcal{O}_s(C)$. Furthermore, by \eqref{equ:max} and \eqref{equ:multiply}, we can derive that
		\begin{align*}
			3^{-2(n-1)/3}\max_{x\in B(y,3^n)}\mathcal{T}^{(1)}_\delta (x)\le \mathcal{O}_s\left(C (\log 3^n)^{\frac{1}{s}}3^{-2(n-1)/3}\right),
		\end{align*}
		which implies the existence of a constant $C(d,s,\delta)>0$ and an exponent $\beta(d,s,\delta)>0$ such that  
		\begin{align*}
			3^{-2(n-1)/3}\max_{x\in B(y,3^n)}\mathcal{T}^{(1)}_\delta (x)\le \mathcal{O}_s\left(C3^{-\beta n}\right).
		\end{align*}			
		Therefore, we apply Lemma \ref{lem:scale} to $ \mathcal{T}_{\delta}^{(2)}(y)$ defined in \eqref{eq.defT2} and obtain 
		\begin{equation}\label{equ:T2_decay}
			\mathcal{T}_{\delta}^{(2)}(y)\le \mathcal{O}_{\beta s}(C).
		\end{equation}
		
		To apply Corollary~\ref{cor.S_max}, we need to define the following minimal scale 
		\begin{align}\label{eq.defR1}
			\mathcal{R}^{(1)}(y) :=\begin{cases}
				\sup\left\{3^{n} : \max_{x\in B(y,3^n)}\mathcal{R}(x)\ge3^{\frac{(n-1)}{3}}  \right\} & \text {on~}\{y\in\cltf\},\\
				0 & \text {on~}\{y\notin\cltf\}.
			\end{cases}
		\end{align}
		Similar to the analysis of $\mathcal{T}_{\delta}^{(2)}(y)$, we can derive that 
		\begin{align}\label{eq.R1_tail}
			\mathcal{R}^{(1)}(y)\le \mathcal{O}_{\beta s}(C).    
		\end{align} 
		In combination, we set
		\begin{equation}\label{eq.defT3}
			\mathcal{T}_\delta ^{(3)}(y) :=\max\left\{\mathcal{T}_\delta^{(2)}(y),\mathcal{R}^{(1)}(y)\right\}. 
		\end{equation}
		This defines the minimal scale in Proposition~\ref{The:main_result_finite}, which satisfies \eqref{eq.TailT3} by applying \eqref{equ:Orlicz_sum} to \eqref{equ:T2_decay} and \eqref{eq.R1_tail}.
		
		Next, we will analyze the typical distance and the atypical distance in succession.

		\medskip

		\textit{Step~2: the typical distance.} 
		Recall the increments in Definition\ref{def.CFH} 
		\begin{align*}
			\Delta S_{t_k} = S_{t_{k+1}} - S_{t_k}, \qquad \Delta \B_{t_k} = \B_{t_{k+1}} - \bar{B}_{t_k}.
		\end{align*}
		By the triangle inequality, the distance on $D$ can be decomposed as 
		\begin{multline}\label{eq.CoarseGrain3}
			\E^\omega\Big[\sup_{t\in[0,T]}\vert S_t-\bar{B}_t\vert\mathbf{1}_D	\Big]\leq  \underbrace{\sum_{k=0}^{n-1}\E^\omega\left[\vert \Delta S_{t_k} - \Delta \B_{t_k}\vert\mathbf{1}_D\right]}_{\text{\eqref{eq.CoarseGrain3}-a}} 
			\\ + \underbrace{\E^\omega\Big[\sup_{0 \leq k \leq n-1}\sup_{t \in I_k}\vert S_t-S_{t_k}\vert \mathbf{1}_D\Big]}_{\text{\eqref{eq.CoarseGrain3}-b}} + \underbrace{\E^\omega\Big[\sup_{0 \leq k \leq n-1}\sup_{t \in I_k}\vert \bar{B}_t-\bar{B}_{t_k}\vert\Big]}_{\text{\eqref{eq.CoarseGrain3}-c}}.
		\end{multline}
		
		For the first term \text{\eqref{eq.CoarseGrain3}-a}, it follows from the conditioned expectation and (B) of Definition~\ref{def.CFH} that
		\begin{equation}\label{equ:E_first_decompose}
			\begin{split}
				&\sum_{k=0}^{n-1}\E^\omega \left[\vert \Delta S_{t_k} - \Delta \B_{t_k}\vert\mathbf{1}_{D}\right] \\
				&= \sum_{k=0}^{n-1}\sum_{x \in B(y,T)}\E^\omega \Ll[\vert \Delta S_{t_k}-\Delta\B_{t_k}\vert \Big\vert S_{t_k} = x \Rr] \P^\omega( S_{t_k} = x) \\
				&= \sum_{k=0}^{n-1}\sum_{x \in B(y,T)}W_1\big(p^\omega(\Delta T,\cdot+x,x), \bar{p}(\Delta T,\cdot)\big) \P^\omega( S_{t_k} = x).
			\end{split}
		\end{equation}
		The conditions $T > \mathcal{T}_\delta ^{(3)}(y)$ and $\Delta T \simeq T^{\frac{2}{3}}$, together with the definition\eqref{eq.defT3} and \eqref{eq.defT2}, imply that
		\begin{align*}
			\forall x \in B(y,T), \qquad \Delta T\ge \mathcal{T}^{(1)}_\delta(x).
		\end{align*}
		Thus, it follows from Proposition~\ref{W_1} that
		\begin{equation*}
			\forall x \in B(y,T), \quad W_1(p^\omega(\Delta T,\cdot+x,x), \bar{p}(\Delta T,\cdot))\le K(\Delta T)^\delta.
		\end{equation*}
		Inserting the above inequality back to \eqref{equ:E_first_decompose}, we can obtain
		\begin{equation}\label{equ:E_first}
			\sum_{k=0}^{n-1}\E^\omega \left[\vert \Delta S_{t_k} - \Delta \B_{t_k}\vert\mathbf{1}_{D}\right]\le K n(\Delta T)^\delta. 
		\end{equation}
		
		\smallskip

		For the second term \text{\eqref{eq.CoarseGrain3}-b}, the condition $T > \mathcal{T}_\delta ^{(3)}(y)$ leads to
		\begin{align*}
			\forall x \in B(y,T), \qquad \Delta T \ge \mathcal{R}(y)^2.
		\end{align*} 
		Therefore, on the event $D$, Corollary~\ref{cor.S_max} immediately shows that 
		\begin{align*}
			\forall k = 0,1,\cdots,n-1, \qquad \sup_{t\in I_k}\vert S_t-S_{t_k}\vert\mathbf{1}_D\le \mathcal{O}^\omega_1(C\sqrt{\Delta T}).
		\end{align*}
		By \eqref{equ:max}, we derive
		\begin{align*}
			\sup_{0 \leq k \leq n-1}\sup_{t\in I_k}\vert S_t-S_{t_k}\vert \mathbf{1}_D \leq \mathcal{O}^\omega_1(C\sqrt{\Delta T} \log n),
		\end{align*}
		which yields
		\begin{equation}\label{equ:E_third}
			\E^\omega\Big[\sup_{0 \leq k \leq n-1}\sup_{t\in I_k}\vert S_t-S_{t_k} \vert \mathbf{1}_D\Big]\le C\sqrt{\Delta T}\log n.
		\end{equation}
		
		\smallskip
		
		The last term \text{\eqref{eq.CoarseGrain3}-c} can be controlled without the minimal scale. The maximal inequality of Brownian motion implies
		\begin{equation}\label{equ:max_Brownian}
			\forall u\ge 0, \quad\P\Big(\sup_{t\in I_k}\vert \bar{B}_t-\bar{B}_{t_k}\vert\ge u\Big)\le 2d \mathbb{P}\Big(\vert \bar{B}_{\Delta T}^1\vert \ge \frac{u}{\sqrt{d}}\Big),
		\end{equation}
		which suggests that $\sup_{t\in I_k}\vert \bar{B}_t-\bar{B}_{t_k}\vert$ has a finite sub-gaussian norm (see \eqref{eq.defOrlicz} and \eqref{eq.defOs} for the definition)
		\begin{align*}
			\Big\Vert \sup_{t\in I_k}\vert \bar{B}_t-\bar{B}_{t_k}\vert \Big\Vert_{\psi_2}\le C\sqrt{\Delta T}.
		\end{align*}
		Using \cite[Exercise 2.5.10]{Vershynin2018High}, it follows that 
		\begin{align}\label{equ:E_second}
			\E^\omega\Big[\sup_{0 \leq k \leq n-1}\sup_{t\in I_k}\vert \bar{B}_t-\bar{B}_{t_k}\vert\Big]\le C\sqrt{\Delta T \log n}.
		\end{align}

		\smallskip
		
		Combining \eqref{eq.CoarseGrain3},\eqref{equ:E_first},\eqref{equ:E_second} and \eqref{equ:E_third}, we obtain 
		\begin{equation}\label{equ:typical}
			\E^\omega\Big[\sup_{t\in[0,T]}\vert S_t-\bar{B}_t\vert\mathbf{1}_D\Big]\le C( n\Delta T^\delta +\sqrt{\Delta T}\log n).
		\end{equation}
		
		\smallskip
		\textit{Step~3: the atypical distance.} 
		For the atypical distance, it follows from the triangle inequality that 
		\begin{equation}\label{eq.atypical_triangle}
			\E^\omega\Big[\sup_{t\in[0,T]}\vert S_t - \bar{B}_t\vert\mathbf{1}_{D^c} \Big]\le\E^\omega\Big[\sup_{t\in[0,T]}\vert S_t - y\vert\mathbf{1}_{D^c} \Big]+\E^\omega\Big[\sup_{t\in[0,T]}\vert \bar{B}_t - y\vert\mathbf{1}_{D^c}\Big].
		\end{equation}
		Since $T>\mathcal{T}_\delta^{(3)}(y)\ge \mathcal{R}(y)$, Lemma~\ref{Pro:Barlow} provides us with the following estimate
		\begin{equation}\label{equ:rare}
			\P^\omega(D^c)\le P^\omega_y\Big(\sup_{t\in[0,T]}\vert S_t\vert > T\Big)\le C\exp(-cT).
		\end{equation}
		Combining Lemma \ref{restrict} with the Gaussian maximal inequality and \eqref{equ:rare}, we can derive
		\begin{equation}\label{equ:rare_B}
			\E^\omega\Big[\sup_{t\in[0,T]}\vert \bar{B}_t - y\vert\mathbf{1}_{D^c}\Big] \le C\sqrt{T}\exp(-c T)(1+T)=o(1).
		\end{equation}
		Similarly, for $T > \mathcal{T}_\delta ^{(3)}(y)$, Corollary~\ref{cor.S_max} applies to $\sup_{t\in[0,T]}\vert S_t - y\vert$. Then, Lemma~\ref{restrict} and \eqref{equ:rare} yield
		\begin{equation}\label{equ:rare_S}
			\E^\omega\Big[\sup_{t\in[0,T]}\vert S_t - y\vert\mathbf{1}_{D^c} \Big]\le C T \exp(-c T)(1+T)=o(1).
		\end{equation}
		
		The above two estimates, together with \eqref{equ:split}, \eqref{equ:typical} and \eqref{eq.atypical_triangle}, imply that there exists a finite positive constant $K(d,\p,\delta)$ such that the following estimate holds when $T > \mathcal{T}_\delta ^{(3)}(y)$
		\begin{equation}\label{equ:result_n}
			\E^\omega\Big[\sup_{[0,T]}\vert S_t-\bar{B}_t\vert\Big] \leq K\left(n (\Delta T)^\delta +\sqrt{\Delta T}\log n\right).
		\end{equation}
		Noting that $n = \lfloor T^{\frac{1}{3}} \rfloor$ and $\Delta T \simeq T^{\frac{2}{3}}$, we obtain the desired result.
	\end{proof}
	
	\begin{proof}[Proof of Proposition~\ref{The:main_result_finite}]
		Definition~\ref{def.CFH} introduces a coupling, Proposition~\ref{Pro:distribution} establishes its consistency, and Proposition~\ref{prop.CFH_bound} provides the desired upper bound.
	\end{proof}

	\subsection{Coupling in infinite horizon}\label{infinite}
	We prove Theorem~\ref{The:main_result} in this subsection using an extension of the coupling in Definition~\ref{def.CFH}.

	\subsubsection{Extension of the coupling}\label{Construction of the coupled process}
	\begin{definition}[Coupling in infinite horizon]\label{def.CIH}
		We define $\CIH$$(y)$ as the coupling $(S_{t},\B_{t})_{t\ge 0}$ in infinite horizon constructed in the following way. Set 
		\begin{equation}\label{equ:T_k}
			T_k :=\frac{1}{2}\left(3^k-1\right),\quad k\in \mathbb{N}.
		\end{equation}
		\begin{enumerate}[label=(\Alph*)]
			\item The processes start from $S_0 = \B_0 = y$.
			\item Conditioned on $(S_t,\B_t)_{t\in [0,T_k]}$, we sample $(\Delta_k S_t, \Delta_k \B_t)_{t\in [0,3^k]}$ satisfying
			\begin{align*} %\label{eq.extend_0}
				\big(S_{T_k} + \Delta_k S_t, S_{T_k} + \Delta_k \B_t\big)_{t\in [0,3^k]} \stackrel{d}{=} \CFH(3^k,\lfloor 3^{\frac{k}{3}}\rfloor, S_{T_k}).
			\end{align*}
			Then we extend the coupled process by 
			\begin{align*} %\label{eq.extend}
				\big(S_{T_k + t},\B_{T_k + t}\big) := \big(S_{T_{k}},\B_{T_{k}}\big) + \big(\Delta_k S_{t}, \Delta_k \B_{t}\big),\quad t\in[0,3^k].
			\end{align*}
		\end{enumerate}
	\end{definition}

	Similar to Proposition~\ref{Pro:distribution}, $\CIH$ is a coupling between the VSRW and Brownian motion with diffusive constant $\bar{\sigma}$ and the proof is omitted here. 
	
	\subsubsection{Distance between $(S_t)_{t \geq 0}$ and $(\B_t)_{t \geq 0}$}\label{Bound the distance} We are now ready to prove the main theorem.

	\begin{proof}[Proof of Theorem 1.1] 
        We prove \eqref{eq.main} using the coupling $\CIH${(y)} defined in Definition~\ref{def.CIH}. 
		
		\textit{Step~1: the minimal scale.}
		Once again, we will use Proposition \ref{The:main_result_finite} to control $\E^\omega\big[\sup_{t\in[0,3^k]}\vert\Delta_k S_t-\Delta_k \B_t\vert\big]$ in the typical case. As noted, we can only apply Proposition \ref{The:main_result_finite} to $\E^\omega\big[\sup_{t\in[0,3^k]}\vert\Delta_k S_t-\Delta_k \B_t\vert\big]$ when $k$ is sufficiently large. Therefore, we need to define a new minimal scale to ensure the estimate in Proposition~\ref{The:main_result_finite}
		\begin{align}\label{eq.defT4}
			\mathcal{T}^{(4)}_\delta(y) :=\begin{cases}
				\sup\left\{3^{n} : \max_{x\in B(y,3^n)}\mathcal{T}_\delta ^{(3)}(x)\ge3^{n-1} \right\} &\text{on~} \{y\in\cltf\},\\
				0 &\text{on~} \{y\notin\cltf\}.
			\end{cases}
		\end{align}
		In line with $\mathcal{T}_\delta^{(2)}(y)$, it follows that there exists an exponent $\beta(d,\p,\delta)>0$ such that
		\begin{equation}\label{equ:T4}
			\mathcal{T}^{(4)}_\delta(y)\le \mathcal{O}_{\beta s}(C).
		\end{equation}
		To account for those terms with small $k$, the final minimal scale is defined as
		\begin{align}\label{eq.defTy}
			\mathcal{T}_\delta(y) :=\left(\mathcal{T}_\delta ^{(4)}(y)\right)^3,
		\end{align}
		which, by \eqref{equ:T4}, satisfies $\mathcal{T}_\delta(y)\le \mathcal{O}_{s}(C)$ with an exponent $s(d,\p,\delta)>0$.
		
		The rest of the proof is under the condition $T>\mathcal{T}_\delta(y)$, which ensures that
		\begin{align}\label{eq.T_final_prop}
			T>\max_{x\in B(y,T)}\mathcal{T}_\delta ^{(3)}(x).
		\end{align}

		\smallskip
		
		\textit{Step~2: the distance.}
		For simplicity, we define two $\omega$-measurable indices
		\begin{equation}\label{eq.defk0k1}
			\begin{split}
				k_0 &:=\inf\left\{k\in\mathbb{N}:T_k>\mathcal{T}^{(4)}_\delta(y)\right\}, \\
				k_1 &:=\inf\left\{k\in\mathbb{N}:T_k>T\right\}.
			\end{split}
		\end{equation}

		The time $T_{k_0}$ serves as the minimal scale to apply Proposition \ref{The:main_result_finite}. Using the triangle inequality and (B) of Definition~\ref{def.CIH}, it follows that
		\begin{equation}\label{eq.CIH_decom}
			\begin{aligned}
				\E^\omega\Big[\sup_{t\in[0,T]}\left\vert S_t-\B_t\right\vert\Big]\le&\E^\omega\Big[\sup_{t\in[0,T_{k_0}]}\left\vert S_t-y\right\vert\Big]+\E^\omega\Big[\sup_{t\in[0,T_{k_0}]}\left\vert\B_t-y\right\vert\Big]\\
				+&\sum_{k=k_{0}}^{k_1-1}\E^\omega\Big[\sup_{t\in[0,3^k]}\left\vert \Delta_k S_t- \Delta_k \B_t\right\vert\Big].
			\end{aligned}
		\end{equation}
		
		For the first term on the {\rhs} of \eqref{eq.CIH_decom}, a simple upper bound for the expectation, based on the number of jumps during $[0,T]$, is sufficient
		\begin{equation}\label{equ:1}
			\E^\omega\Big[\sup_{t\in[0,T_{k_0}]}\left\vert S_t-y\right\vert\Big]\le 2d T_{k_0}.
		\end{equation}
		For the second term, by the Gaussian maximal inequality and integration by parts, we have
		\begin{equation}\label{equ:2}
			\E^\omega\Big[\sup_{t\in[0,T_{k_0}]}\left\vert \B_t-y\right\vert\Big]\le C\sqrt{ T_{k_0}}.
		\end{equation}
		For each summand in the third term on the {\rhs} of \eqref{eq.CIH_decom}, we can use the same technique as in \eqref{equ:split} to split the expectation into contributions from the typical case and the atypical case
		\begin{equation}\label{equ:infty_triangle2}
			\begin{aligned}
				\E^\omega\Big[\sup_{t\in[0,3^k]}\left\vert \Delta_k S_t-\Delta_k \B_t \right\vert\Big]
				&\le \E^\omega\Big[\sup_{t\in[0,3^k]}\left\vert \Delta_k S_t-\Delta_k \B_t\right\vert\mathbf{1}_{\left\{S_{T_k}\in B(y,T_k)\right\}}\Big]\\
				& \quad+\E^\omega\Big[\sup_{t\in[0,3^k]}\left\vert \Delta_k S_t\right\vert\mathbf{1}_{\left\{S_{T_k} \notin B(y,T_k)\right\}}\Big]\\
				& \quad+\E^\omega\Big[\sup_{t\in[0,3^k]}\left\vert \Delta_k \B_t\right\vert\mathbf{1}_{\left\{S_{T_k} \notin B(y,T_k)\right\}}\Big].
			\end{aligned}
		\end{equation}
		
		For the typical case, it follows from the conditioned expectation and (B) of Definition~\ref{def.CIH} that 
		\begin{equation}\label{equ:infty_typical}
			\begin{split}
				&\E^\omega\Big[\sup_{t\in[0,3^k]}\left\vert \Delta_k S_t-\Delta_k \B_t\right\vert\mathbf{1}_{\left\{S_{T_k}\in B(y,T_k)\right\}}\Big]\\
				&=\sum_{x \in B(y,T_k)}\E^\omega\Big[\sup_{t\in[0,3^k]}\left\vert \Delta_k S_t-\Delta_k \B_t\right\vert  \Big \vert S_{T_k} = x \Big] \P^\omega\Ll(S_{T_k} = x\Rr).
			\end{split}
		\end{equation}
		By (A) of Definition~\ref{def.CIH}, $( x + \Delta_k S_t, x +\Delta_k \B_t)_{t \in [0,3^k]}$ is sampled from $\CFH$($3^k,\lfloor 3^{\frac{k}{3}}\rfloor, x$). 
		Furthermore, from \eqref{eq.defT4} and \eqref{eq.defk0k1}, we have
		\begin{align*}
			\forall k\ge k_0, \qquad 3^k>\max_{x\in B(y,T_k)}\mathcal{T}_\delta ^{(3)}(x).
		\end{align*}
		Thus, Proposition~\ref{The:main_result_finite} implies that 
		\begin{equation}
			\forall k\ge k_0, x \in B(y,T_k), \qquad \E^\omega\Big[\sup_{t\in[0,3^k]}\left\vert \Delta_k S_t-\Delta_k \B_t\right\vert \Big \vert S_{T_k} = x \Big]\le K 3^{(\frac{1}{3}+\delta)k}.
		\end{equation}
		Inserting the above inequality back to \eqref{equ:infty_typical}, we can obtain
		\begin{equation}\label{eq.infty_typical}
			\forall k\ge k_0,\qquad\E^\omega\Big[\sup_{t\in[0,3^k]}\left\vert \Delta_k S_t-\Delta_k \B_t\right\vert\mathbf{1}_{\left\{S_{T_k}\in B(y,T_k)\right\}}\Big]\le K3^{(\frac{1}{3}+\delta)k}.
		\end{equation}
		\smallskip
		
		For the atypical case, by $T_k\ge\mathcal{T}_\delta ^{(3)}(y)>\mathcal{R}(y)$, it follows from Lemma~\ref{Pro:Barlow} that 
		\begin{equation*}
			\forall k\ge k_0,\qquad\P^\omega\big(S_{T_k}\notin B(y,T_k)\big)\le P^\omega_y\Big(\sup_{t\in[0,T_k]}\vert S_t\vert>T_k\Big)\le C\exp\left(-c T_k\right).
		\end{equation*}
		Therefore, by the same reasoning as in \eqref{equ:rare_B} and \eqref{equ:rare_S}, for large $k$ we have 
		\begin{equation}\label{equ:infty_remaining}
			\begin{aligned}
				&\E^\omega\Big[\sup_{t\in[0,3^k]}\left\vert \Delta_k S_t\right\vert\mathbf{1}_{\left\{S_{T_k}\notin B(y,T_k)\right\}}\Big]= o(1),\\
				&\E^\omega\Big[\sup_{t\in[0,3^k]}\left\vert \Delta_k \B_t\right\vert\mathbf{1}_{\left\{S_{T_k}\notin B(y,T_k)\right\}}\Big]= o(1).
			\end{aligned}
		\end{equation}
		Together with \eqref{equ:infty_triangle2},\eqref{eq.infty_typical} and \eqref{equ:infty_remaining}, we can derive that
		\begin{equation}\label{equ:3}
			\E^\omega\Big[\sup_{t\in[0,3^k]}\left\vert \Delta_k S_t-\Delta_k \B_t \right\vert\Big]\le K3^{(\frac{1}{3}+\delta)k}.
		\end{equation}

		Combining \eqref{eq.CIH_decom}, \eqref{equ:1},\eqref{equ:2} and \eqref{equ:3}, we obtain 
		\begin{equation}\label{eq.pre_conclusion}
			\E^\omega\Big[\sup_{t\in[0,T]}\left\vert S_t-\B_t\right\vert\Big]\le CT_{k_0}+KT_{k_1}^{\frac{1}{3}+\delta}.
		\end{equation}
		Finally, we need to control $T_{k_0}$ and $T_{k_1}$. It follows from the definition of \eqref{eq.defk0k1} that 
		\begin{align}\label{eq.Tk1}
			T < T_{k_1}\le 3T, \qquad \mathcal{T}^{(4)}_\delta(y) < T_{k_0} \leq 3 \mathcal{T}^{(4)}_\delta(y).
		\end{align}
		Combining this with $\mathcal{T}_\delta(y)$ defined in \eqref{eq.defTy} and the condition $T > \mathcal{T}_\delta(y)$, we can derive that
		\begin{align}\label{eq.Tk2}
			T_{k_0}^3\le \Ll(3 \mathcal{T}^{(4)}_\delta(y)\Rr)^3 = 27 \mathcal{T}_\delta(y) \leq 27 T.
		\end{align}
		We substitute \eqref{eq.Tk1} and \eqref{eq.Tk2} back into \eqref{eq.pre_conclusion} and conclude that
		\begin{align}
			\E^\omega\Big[\sup_{t\in[0,T]}\left\vert S_t-\B_t\right\vert\Big] \le KT^{\frac{1}{3}+\delta}.
		\end{align}
		
	\end{proof}

	\section{Identification of limit constant in LIL}\label{sec.LIL}
	In this section, we will utilize Theorem \ref{The:main_result} to prove Corollary \ref{Cor:LIL_constant}. First, we recall the LIL for standard Brownian motion under the $\ell_p$-norm for $1\le p<\infty$. For details, see \cite[Corollary~2]{kuelbs1973law}.

	\begin{lemma}\label{Pro:LIL}
		Given $d$-dimensional standard Brownian motion $(B_t)_{t \geq 0}$, we have
		\begin{align}\label{eq.LIL_BM}
			\mathbb{P}\Big(\limsup_{t\to \infty}\frac{\vert B_t\vert_p }{\sqrt{2t\log\log t}}= \gamma_{d,p}\Big)=1.
		\end{align}
		Here the constant $\gamma_{d,p}$ is defined in \eqref{eq.defLIL_para} and is characterized by 
		\begin{align}\label{eq.defgamme}
			\gamma_{d,p} = \sup_{x \in \R^d: \vert x\vert_2=1}\vert x\vert_p .
		\end{align}
	\end{lemma}
	\begin{remark}
		The corresponding statement for the $\ell_2$-norm can also be found in \cite[Exercise 1.21 of Chapter 2]{revuz2013continuous}. 
		Additionally, the characterization in \eqref{eq.defgamme} is an elementary exercise:
		\begin{itemize}[label=---]
			\item For $1 \leq p \leq 2$, by Jensen's inequality, we have
			\begin{equation*}
				\vert x\vert_p \leq d^{\frac{1}{p}-\frac{1}{2}}|x|_2=d^{\frac{1}{p}-\frac{1}{2}},
			\end{equation*} 
			where the equality attained when $x = (\frac{1}{\sqrt{d}}, \frac{1}{\sqrt{d}}, \cdots, \frac{1}{\sqrt{d}})$.
			\item For $2 \leq p < \infty$, we have $\vert x \vert_p \leq \vert x \vert_2 = 1$, where the equality is attained when $x = (1,0,0, \cdots, 0)$.  
		\end{itemize}
		
	\end{remark}

	\begin{proof}[Proof of Corollary 1.2]
		For convenience, we use the following shorthand in the proof
		\begin{align*}
			\phi(t):=2t\log\log t.
		\end{align*}
		The proof is divided into two parts: VSRW and CSRW. 
		
		\textit{Step~1: LIL for VSRW.} We prove the statement \eqref{eq.LIL_VSRW} using the coupling $(S_t, \B_t)_{t \geq 0}$ in Theorem~\ref{The:main_result}.
		Using \eqref{eq.LIL_BM}, we can derive that
		\begin{align*}
			\P\Big(\limsup_{t\to\infty }\frac{\vert \bar{B}_t\vert_p}{\sqrt{\phi(t)}} =\gamma_{d,p}\bar{\sigma}\Big)=\P\Big( \limsup_{t\to \infty}\frac{\vert  B_t\vert_p }{\sqrt{\phi(t)}} = \gamma_{d,p}\Big)=1.
		\end{align*}
		By the triangle inequality, it follows that
		\begin{align*}
			\vert \bar{B}_t\vert_p-\vert S_t-\bar{B}_t\vert_p\le\vert S_t\vert_p
			\le \vert \bar{B}_t\vert_p+\vert S_t-\bar{B}_t\vert_p.
		\end{align*}
		Therefore, to prove Corollary \ref{Cor:LIL_constant}, it is sufficient to show that
		\begin{equation}\label{equ:limsup}
			\P^\omega\Big(\limsup_{t\to \infty}\frac{\vert S_t-\bar{B}_t\vert_p}{\sqrt{\phi(t)}}=0\Big)=1.
		\end{equation}
		
		Since all the vector norms in $\R^d$ are equivalent, it follows that $\vert \cdot \vert_p \leq C_{d,p} \vert \cdot \vert_2$.  Then, for every $\epsilon>0$, Theorem \ref{The:main_result} and the Markov inequality imply that, for sufficiently large $k$,
		\begin{align*}
			\P^\omega\Big(\sup_{t\in[T_k,T_{k+1}]}\vert S_t-\bar{B}_t\vert_p>\epsilon \sqrt{\phi(T_k)}\Big)&\le \P^\omega\Big(\sup_{t\in[0,T_{k+1}]} C_{d,p}\vert S_t-\bar{B}_t\vert_2>\epsilon \sqrt{\phi(T_k)}\Big)\\
			&\le \frac{KT_{k+1}^{\frac{1}{3}+\delta}}{\epsilon C_{d,p}\sqrt{\phi(T_k)}},
		\end{align*}
		where $T_k=\frac{1}{2}(3^k-1)$ as defined in \eqref{equ:T_k}.
		Because $2 T_k < T_{k+1} < 3T_k$ and $\delta$ can be made arbitrarily small, we derive that
		\begin{equation*}
			\sum_{k=1}^{\infty}\frac{T_{k+1}^{\frac{1}{3}+\delta}}{\epsilon C_{d,p}\sqrt{\phi(T_k)}}\le K\sum_{k=1}^{\infty}\frac{T_k^{\frac{1}{3}+\delta}}{\epsilon C_{d,p}\sqrt{\phi(T_k)}}\le (\epsilon  C_{d,p})^{-1}K\sum_{k=1}^{\infty}T_k^{-\frac{1}{6}+\delta}<\infty.
		\end{equation*}
		By the Borel--Cantelli lemma, we have
		\begin{equation}\label{eq.borel-Cantelli}
			\P^\omega\Big(\limsup_{t\to \infty}\frac{\vert S_t-\bar{B}_t\vert_p}{\sqrt{\phi(t)}}\le \epsilon\Big)=1.
		\end{equation}
		Since \eqref{eq.borel-Cantelli} holds and $\epsilon>0$ can be arbitrarily small, this indicates that \eqref{equ:limsup} also holds. Thus, on the event $\{y\in \cltf\}$, we conclude that
		\begin{equation*}
			P^\omega_y\Big(\limsup_{t\to\infty}\frac{\vert S_t\vert_p}{\sqrt{\phi(t)}}=\gamma_{d,p} \bar{\sigma}\Big)=1.
		\end{equation*}

		\smallskip
		
		\textit{Step~2: LIL for CSRW.} The transition from VSRW to CSRW requires a time change of the process, which is classical; we refer to \cite[Section~6.2]{ABDH}. For each $x \in \mathbb{Z}^{d}$, let $\tau_{x}: \Omega \rightarrow \Omega$ be the ``shift by $x$'' defined by $\left(\tau_{x} \omega\right)_{e}=\omega_{x+e}$. Recall that $\omega_x=\sum_{z:z\sim x}\omega\left(\{x,z\}\right)$. We set $F(\omega) :=\omega_0$ and define
		\begin{equation*}
			A_t :=\int_{0}^{t}\omega_{S_u} \,\d u=\int_0^t F(\tau_{S_u}\omega)\,\d u.
		\end{equation*}
		We define the inverse of $A_t$ as
		\begin{align*}
			\mathfrak{a}_t :=\inf \{s\ge 0: A_s\ge t\},
		\end{align*}
		then the time-changed process 
		\begin{equation}\label{equ:time_changed}
			X_t := S_{\mathfrak{a}_t}
		\end{equation}
		follows the law of CSRW. 
		
		Let $\P_y$ be the conditioned probability measure defined as 
		\begin{equation*}
			\mathbb{P}_{y}(A):=\mathbb{P}_\p\left(A \mid y\in\cltf\right), \quad A \in \mathcal{F}.
		\end{equation*}
		By the equation between (6.5) and (6.6) of \cite{ABDH}, along with the shift invariance property of percolation, we have 
		\begin{equation}\label{equ:ergodic}
			\lim_{t\to \infty}\frac{A_t}{t}=2d\E_\p\left[\omega\left(\{0,x\}\right)\vert 0\in\cltf\right],\qquad \P_y\times P^\omega_y- a.s.,
		\end{equation}
		where $x$ is a neighboring vertex of $0$.
		
		The time change formula \eqref{equ:ergodic} implies
		\begin{equation}\label{equ:inverse}
			\lim_{t\to\infty}\sqrt{\frac{\phi(\mathfrak{a}_t)}{\phi(t)}}= \left(2d\E_\p\left[\omega\left(\{0,x\}\right)\vert 0\in\cltf\right]\right)^{-\frac{1}{2}} = \alpha_{d,\p}, \quad \P_y\times P^\omega_y- a.s..
		\end{equation}
		It is worth noting that \eqref{equ:inverse} is a quenched result.
		
		Combining \eqref{equ:time_changed} and \eqref{equ:inverse}, we show that, on the event $\{y\in\cltf\}$,
		\begin{align*}
			&P^\omega_y \Big(\limsup_{t\to\infty}\frac{\vert \tilde{S_t}\vert_p}{\sqrt{\phi(t)}}= \alpha_{d,\p}\gamma_{d,p} \bar{\sigma}\Big)\\
			=&P^\omega_y\Big(\limsup_{t\to\infty}\frac{\vert S_{\mathfrak{a}_t}\vert_p}{\sqrt{\phi(\mathfrak{a}_t)}}\sqrt{\frac{\phi(\mathfrak{a}_t)}{\phi(t)}}=\alpha_{d,\p}\gamma_{d,p} \bar{\sigma}\Big)\\
			=&P^\omega_y\Big(\limsup_{t\to\infty}\frac{\vert S_{\mathfrak{a}_t}\vert_p}{\sqrt{\phi(\mathfrak{a}_t)}}=\gamma_{d,p} \bar{\sigma}\Big)=1,
		\end{align*}
		which concludes the proof.
	\end{proof}
	
	\appendix
	\section{The homogenized matrix is a scalar matrix}
	The homogenized matrix $\ab$ in \cite[Definition~5.1]{armstrong2018elliptic} is a positive diagonal matrix. Here, we provide a detailed proof, which arises from the symmetry in the description of the variational formula. Thus,  we state a more general version. The notation $C_0(\cu_{m})$ stands for the functions taking value $0$ on $\partial \cu_m$.
	
	\begin{lemma}\label{lem.scalar}
		Let $\{\omega(\{x,y\})\}_{x,y \in \Zd}$ be a positive random conductance with law invariant under permutations and reflections of coordinates. If the homogenized matrix
		\begin{align*}
			q \cdot \ab q  := \lim_{m \to \infty} \E\Big[  \inf_{v \in \ell_{q} +  C_0(\cu_{m})}  \frac{1}{\vert \cu_{m} \vert}\sum_{x,y\in \cu_{m}}\frac{1}{2} \omega(\{x,y\}) (v(y)-v(x))^2\Big],
		\end{align*}
		is well-defined, then there exists a scalar constant $C^* \geq 0$ such that 
		\begin{align*}
			\ab = C^* \mathrm{Id}.
		\end{align*}
	\end{lemma}
	\begin{proof}
		Let $\hat{x}$ denote permutations or reflections of coordinates where $x\in\R^d$, and define $\hat{f}(x):=f(\hat{x})$. 
        By the change of variable and the symmetry of $\square_m$, it follows that 
		\begin{align*}
			q \cdot \ab q  &= \lim_{m \to \infty} \E\Big[  \inf_{v \in \ell_{q} + C_0(\cu_{m})}  \frac{1}{\vert \cu_{m} \vert}\sum_{x,y\in \cu_{m}}\frac{1}{2} \omega(\{x,y\}) (v(y)-v(x))^2\Big]\\
			&= \lim_{m \to \infty} \E\Big[  \inf_{v \in \ell_{q} + C_0(\cu_{m})}  \frac{1}{\vert \cu_{m} \vert}\sum_{x,y\in \cu_{m}}\frac{1}{2} \omega(\{\hat{x},\hat{y}\}) (v(\hat{y})-v(\hat{x}))^2\Big]\\
			&= \lim_{m \to \infty} \E\Big[  \inf_{\hat{v} \in \ell_{q} + C_0(\cu_{m})}  \frac{1}{\vert \cu_{m} \vert}\sum_{x,y\in \cu_{m}}\frac{1}{2} \omega(\{x,y\})  (\hat{v}(y)-\hat{v}(x))^2\Big]\\
			&= \hat{q} \cdot \ab \hat{q},
		\end{align*}
		where the third equality follows from the definition $\hat{v}(x):=v(\hat{x})$ and the invariance of law $\{\omega(\{x,y\})\}_{x,y \in \Zd}$. 
		
		We further implement polarization and obtain 
		\begin{align}\label{eq.IdPolar}
			q'\cdot \ab q = \hat{q'} \cdot \ab \hat{q},\qquad\forall q,q'\in\R^d.
		\end{align}
		Let $\{e_i\}_{1 \leq i \leq d}$ be the canonical basis in $\R^d$. Taking $q' = e_i, q = e_j, i \neq j$, and letting $\hat{x}$ be the reflection of the $i$-th coordinate in \eqref{eq.IdPolar}, we have
		\begin{align*}
			\ab_{ij} = -\ab_{ij}, \qquad \forall 1 \leq i \neq j \leq d, 
		\end{align*}
		which implies that $\ab$ is diagonal.
		
		Moreover, we take $q=q'=e_i$ and $\hat{x}$ as the permutation between the $i$-th and $j$-th coordinates ($i \neq j$) in \eqref{eq.IdPolar}, which implies 
		\begin{align*}
			\ab_{ii} = \ab_{jj}, \qquad \forall 1 \leq i \neq j \leq d.
		\end{align*}
		This concludes that $\ab = C^* \mathrm{Id}$ with $C^* \geq 0$. 
	\end{proof}
	\section{\texorpdfstring{The existence of the coupling $\CFH$}{The existence of the coupling CFH}}
	\begin{lemma}\label{lem.existence}
		The coupling $\CFH$$(T,n,y)$ in Definition~\ref{def.CFH} is realizable. 
	\end{lemma}
	\begin{proof}
		 It suffices to justify (B) in Definition~\ref{def.CFH}, which is a nontrivial step. 
		 Assuming that $\big(S_{t_i},\B_{t_i}\big)_{0\le i\le k}$ has already been constructed, we sample random variables $\big\{\big(\Delta S_{t_k}^{(x)}, \Delta\B^{(x)}_{t_k}\big) : x \in \cltf\big\}$ such that the following conditions hold.
		\begin{itemize}[label=---]
			\item For every $x \in \cltf$,  the vector $\big(\Delta S_{t_k}^{(x)}, \Delta\B^{(x)}_{t_k}\big)$ represents the coupling of the increment  $\Ll(p^\omega(\Delta T,x +\cdot,x), \bar{p}(\Delta T,\cdot) \Rr)$ to minimize the $W_1$ distance, i.e.,
			%\begin{align*}
				%\big(\Delta S_{t_k}^{(x)}, \Delta\B^{(x)}_{t_k}\big)  \in \varPi\big(p^\omega(\Delta T,x +\cdot,x), \bar{p}(\Delta T,\cdot) \big),
			%\end{align*}
			%and 
			\begin{equation*} %\label{eq.defDeltaSPre}
				\E^\omega \big[\vert \Delta S_{t_k}^{(x)}-\Delta\B^{(x)}_{t_k}\vert \big]=W_1\big(p^\omega(\Delta T,x+\cdot,x), \bar{p}(\Delta T,\cdot) \big).
			\end{equation*}
			\item The couplings $\{\big(\Delta S_{t_k}^{(x)}, \Delta\B^{(x)}_{t_k}\big)\}_{ x \in \cltf}$ are independent and are also independent of $\big(S_{t_i},\B_{t_i}\big)_{0\le i\le k}$.
		\end{itemize}
		The existence of the optimal coupling is guaranteed; see Definition~\ref{def.Wp}. We then define the increment
		\begin{equation}\label{eq.defDeltaS}
			\big(\Delta S_{t_k}, \Delta\B_{t_k}\big) := \sum_{x \in \cltf} \big(\Delta S_{t_k}^{(x)}, \Delta\B^{(x)}_{t_k}\big) \Ind{S_{t_k} = x}.
		\end{equation}
		Conditioned on $\big(S_{t_i},\B_{t_i}\big)_{0\le i\le k}$, it follows directly that
        \begin{equation*}
            \big(\Delta S_{t_k}, \Delta\B_{t_k}\big)  \in \varPi\big(p^\omega(\Delta T,S_{t_k} +\cdot,S_{t_k}), \bar{p}(\Delta T,\cdot) \big).
        \end{equation*}
        Furthermore, by the definition \eqref{eq.defDeltaS} and the independence property, we can obtain
		\begin{align*}
			&\E^\omega \big[\vert \Delta S_{t_k}-\Delta\B_{t_k}\vert \big\vert \big(S_{t_i},\B_{t_i}\big)_{0\le i\le k} \big] \\
			&=\sum_{x\in\cltf}\E^\omega \big[\vert \Delta S_{t_k}^{(x)} - \Delta\B^{(x)}_{t_k}\vert\big\vert \big(S_{t_i},\B_{t_i}\big)_{0\le i\le k} \big]\Ind{S_{t_k} = x}  \\
			&=\sum_{x\in\cltf}W_1\big(p^\omega(\Delta T,x+\cdot,x), \bar{p}(\Delta T,\cdot) \big)\Ind{S_{t_k} = x}\\
			&=W_1\big(p^\omega(\Delta T,S_{t_k}+\cdot,S_{t_k}), \bar{p}(\Delta T,\cdot)\big),
		\end{align*}
        which concludes the proof.
	\end{proof}

	\section*{Acknowledgements}
	We thank the anonymous referees for their careful reading of the manuscript and their insightful comments and suggestions, which have significantly improved the quality of this paper. We also thank Yanxin Zhou and Wenhao Zhao for discussions in the early stage of this project. Gu was supported by NSFC (No.12301166). Su was partially supported by NSFC (Nos.12271475 and U23A2064).

	\bibliographystyle{abbrv}
	\bibliography{Ref}

\begin{thebibliography}{10}

\bibitem{aizenman1987uniqueness}
M.~Aizenman, H.~Kesten, and C.~M. Newman.
\newblock Uniqueness of the infinite cluster and continuity of connectivity
  functions for short and long range percolation.
\newblock {\em Comm. Math. Phys.}, 111(4):505--531, 1987.

\bibitem{AKT84}
M.~Ajtai, J.~Koml\'os, and G.~Tusn\'ady.
\newblock On optimal matchings.
\newblock {\em Combinatorica}, 4(4):259--264, 1984.

\bibitem{ABDH}
S.~Andres, M.~T. Barlow, J.-D. Deuschel, and B.~M. Hambly.
\newblock Invariance principle for the random conductance model.
\newblock {\em Probab. Theory Related Fields}, 156(3-4):535--580, 2013.

\bibitem{andres2015invariance}
S.~{Andres}, J.-D. {Deuschel}, and M.~{Slowik}.
\newblock {Invariance principle for the random conductance model in a
  degenerate ergodic environment.}
\newblock {\em {Ann. Probab.}}, 43(4):1866--1891, 2015.

\bibitem{armstrong2018elliptic}
S.~{Armstrong} and P.~{Dario}.
\newblock {Elliptic regularity and quantitative homogenization on percolation
  clusters.}
\newblock {\em {Commun. Pure Appl. Math.}}, 71(9):1717--1849, 2018.

\bibitem{armstrong2022elliptic}
S.~Armstrong and T.~Kuusi.
\newblock Elliptic homogenization from qualitative to quantitative.
\newblock {\em arXiv preprint arXiv:2210.06488}, 2022.

\bibitem{armstrong2016mesoscopic}
S.~Armstrong, T.~Kuusi, and J.-C. Mourrat.
\newblock Mesoscopic higher regularity and subadditivity in elliptic
  homogenization.
\newblock {\em Comm. Math. Phys.}, 347(2):315--361, 2016.

\bibitem{armstrong2017additive}
S.~Armstrong, T.~Kuusi, and J.-C. Mourrat.
\newblock The additive structure of elliptic homogenization.
\newblock {\em Invent. Math.}, 208(3):999--1154, 2017.

\bibitem{AKMbook}
S.~Armstrong, T.~Kuusi, and J.-C. Mourrat.
\newblock {\em Quantitative stochastic homogenization and large-scale
  regularity}, volume 352 of {\em Grundlehren der mathematischen
  Wissenschaften}.
\newblock Springer Nature, 2019.

\bibitem{armstrong2016lipschitz}
S.~N. Armstrong and J.-C. Mourrat.
\newblock Lipschitz regularity for elliptic equations with random coefficients.
\newblock {\em Arch. Ration. Mech. Anal.}, 219(1):255--348, 2016.

\bibitem{armstrong2016quantitative}
S.~N. Armstrong and C.~K. Smart.
\newblock Quantitative stochastic homogenization of convex integral
  functionals.
\newblock {\em Ann. Sci. \'{E}c. Norm. Sup\'{e}r. (4)}, 49(2):423--481, 2016.

\bibitem{barlow2004random}
M.~T. Barlow.
\newblock Random walks on supercritical percolation clusters.
\newblock {\em Ann. Probab.}, 32(4):3024--3084, 2004.

\bibitem{BD}
M.~T. Barlow and J.-D. Deuschel.
\newblock Invariance principle for the random conductance model with unbounded
  conductances.
\newblock {\em Ann. Probab.}, 38(1):234--276, 2010.

\bibitem{hambly2009parabolic}
M.~T. Barlow and B.~M. Hambly.
\newblock Parabolic {H}arnack inequality and local limit theorem for
  percolation clusters.
\newblock {\em Electron. J. Probab.}, 14:no. 1, 1--27, 2009.

\bibitem{bartfai1966bestimmung}
P.~B{\'a}rtfai.
\newblock Die bestimmung der zu einem wiederkehrenden prozess geh{\"o}renden
  verteilungsfunktion aus den mit fehlern behafteten daten einer einzigen
  realisation.
\newblock {\em Studia Sci. Math. Hungar}, 1:161--168, 1966.

\bibitem{benjamini2003}
I.~Benjamini and E.~Mossel.
\newblock On the mixing time of a simple random walk on the super critical
  percolation cluster.
\newblock {\em Probab. Theory Related Fields}, 125(3):408--420, 2003.

\bibitem{berger2007quenched}
N.~Berger and M.~Biskup.
\newblock Quenched invariance principle for simple random walk on percolation
  clusters.
\newblock {\em Probab. Theory Related Fields}, 137(1-2):83--120, 2007.

\bibitem{biskup2011recent}
M.~Biskup.
\newblock Recent progress on the random conductance model.
\newblock {\em Probab. Surv.}, 8:294--373, 2011.

\bibitem{thomas2020}
T.~Bonis.
\newblock Stein's method for normal approximation in {W}asserstein distances
  with application to the multivariate central limit theorem.
\newblock {\em Probab. Theory Related Fields}, 178(3-4):827--860, 2020.

\bibitem{bou2023rigidity}
A.~Bou-Rabee, W.~Cooperman, and P.~Dario.
\newblock Rigidity of harmonic functions on the supercritical percolation
  cluster.
\newblock {\em arXiv preprint arXiv:2303.04736}, 2023.

\bibitem{broadbent1957percolation}
S.~R. Broadbent and J.~M. Hammersley.
\newblock Percolation processes: I. crystals and mazes.
\newblock In {\em Mathematical proceedings of the Cambridge philosophical
  society}, volume~53, pages 629--641. Cambridge University Press, 1957.

\bibitem{burton1989density}
R.~M. Burton and M.~Keane.
\newblock Density and uniqueness in percolation.
\newblock {\em Comm. Math. Phys.}, 121(3):501--505, 1989.

\bibitem{chatterjee12}
S.~Chatterjee.
\newblock A new approach to strong embeddings.
\newblock {\em Probab. Theory Related Fields}, 152(1-2):231--264, 2012.

\bibitem{chen2024quenched}
X.~Chen, T.~Kumagai, and J.~Wang.
\newblock Quenched local limit theorem for random conductance models with
  long-range jumps.
\newblock {\em arXiv preprint arXiv:2402.07212}, 2024.

\bibitem{CFP19}
T.~A. Courtade, M.~Fathi, and A.~Pananjady.
\newblock Existence of {S}tein kernels under a spectral gap, and discrepancy
  bounds.
\newblock {\em Ann. Inst. Henri Poincar\'e{} Probab. Stat.}, 55(2):777--790,
  2019.

\bibitem{csorgo1981strong}
M.~Cs\"org{\H o} and P.~R\'ev\'esz.
\newblock {\em Strong approximations in probability and statistics}.
\newblock Probability and Mathematical Statistics. Academic Press, Inc.
  [Harcourt Brace Jovanovich, Publishers], New York-London, 1981.

\bibitem{csorgo1984KMT}
S.~Cs\"org{\H o} and P.~Hall.
\newblock The {K}oml\'os-{M}ajor-{T}usn\'ady approximations and their
  applications.
\newblock {\em Austral. J. Statist.}, 26(2):189--218, 1984.

\bibitem{dario2021corrector}
P.~Dario.
\newblock Optimal corrector estimates on percolation cluster.
\newblock {\em Ann. Appl. Probab.}, 31(1):377--431, 2021.

\bibitem{dario2021quantitative}
P.~Dario and C.~Gu.
\newblock Quantitative homogenization of the parabolic and elliptic {G}reen's
  functions on percolation clusters.
\newblock {\em Ann. Probab.}, 49(2):556--636, 2021.

\bibitem{Davies1993Large}
E.~B. Davies.
\newblock Large deviations for heat kernels on graphs.
\newblock {\em J. London Math. Soc. (2)}, 47(1):65--72, 1993.

\bibitem{dimitrov2021}
E.~Dimitrov and X.~Wu.
\newblock K{MT} coupling for random walk bridges.
\newblock {\em Probab. Theory Related Fields}, 179(3-4):649--732, 2021.

\bibitem{duminil2008law}
H.~Duminil-Copin.
\newblock Law of the iterated logarithm for the random walk on the infinite
  percolation cluster.
\newblock {\em arXiv preprint arXiv:0809.4380}, 2008.

\bibitem{EMZ2020}
R.~Eldan, D.~Mikulincer, and A.~Zhai.
\newblock The {CLT} in high dimensions: quantitative bounds via martingale
  embedding.
\newblock {\em Ann. Probab.}, 48(5):2494--2524, 2020.

\bibitem{FG15}
N.~Fournier and A.~Guillin.
\newblock On the rate of convergence in {W}asserstein distance of the empirical
  measure.
\newblock {\em Probab. Theory Related Fields}, 162(3-4):707--738, 2015.

\bibitem{gloria2015quantification}
A.~Gloria, S.~Neukamm, and F.~Otto.
\newblock Quantification of ergodicity in stochastic homogenization: optimal
  bounds via spectral gap on {G}lauber dynamics.
\newblock {\em Invent. Math.}, 199(2):455--515, 2015.

\bibitem{gloria2011optimal}
A.~Gloria and F.~Otto.
\newblock An optimal variance estimate in stochastic homogenization of discrete
  elliptic equations.
\newblock {\em Ann. Probab.}, 39(3):779--856, 2011.

\bibitem{gloria2012optimal}
A.~Gloria and F.~Otto.
\newblock An optimal error estimate in stochastic homogenization of discrete
  elliptic equations.
\newblock {\em Ann. Appl. Probab.}, 22(1):1--28, 2012.

\bibitem{GO3}
A.~Gloria and F.~Otto.
\newblock Quantitative results on the corrector equation in stochastic
  homogenization.
\newblock {\em J. Eur. Math. Soc. (JEMS)}, 19(11):3489--3548, 2017.

\bibitem{grimmett1999}
G.~Grimmett.
\newblock {\em Percolation}, volume 321 of {\em Grundlehren der Mathematischen
  Wissenschaften [Fundamental Principles of Mathematical Sciences]}.
\newblock Springer-Verlag, Berlin, second edition, 1999.

\bibitem{Gu2020uniform}
C.~Gu.
\newblock Uniform estimate of an iterative method for elliptic problems with
  rapidly oscillating coefficients.
\newblock {\em Stoch. Partial Differ. Equ. Anal. Comput.}, 8(4):787--818, 2020.

\bibitem{gu2022efficient}
C.~Gu.
\newblock An efficient algorithm for solving elliptic problems on percolation
  clusters.
\newblock {\em The Annals of Applied Probability}, 32(4):2755--2810, 2022.

\bibitem{HHP06}
C.~Hoffman, A.~E. Holroyd, and Y.~Peres.
\newblock A stable marriage of {P}oisson and {L}ebesgue.
\newblock {\em Ann. Probab.}, 34(4):1241--1272, 2006.

\bibitem{kesten1982percolation}
H.~Kesten.
\newblock {\em Percolation theory for mathematicians}, volume~2 of {\em
  Progress in Probability and Statistics}.
\newblock Birkh\"auser, Boston, MA, 1982.

\bibitem{KMT1}
J.~Koml\'{o}s, P.~Major, and G.~Tusn\'{a}dy.
\newblock An approximation of partial sums of independent {RV}'s, and the
  sample {DF}. {II}.
\newblock {\em Z. Wahrscheinlichkeitstheorie und Verw. Gebiete}, 34(1):33--58,
  1976.

\bibitem{kozlov1985}
S.~M. Kozlov.
\newblock The averaging method and walks in inhomogeneous environments.
\newblock {\em Uspekhi Mat. Nauk}, 40(2(242)):61--120, 238, 1985.

\bibitem{kuelbs1973law}
J.~Kuelbs and R.~LePage.
\newblock The law of the iterated logarithm for brownian motion in a banach
  space.
\newblock {\em Transactions of the American Mathematical Society},
  185:253--264, 1973.

\bibitem{kumagai2016}
T.~Kumagai and C.~Nakamura.
\newblock Laws of the iterated logarithm for random walks on random conductance
  models.
\newblock In {\em Stochastic analysis on large scale interacting systems},
  volume B59 of {\em RIMS K\^{o}ky\^{u}roku Bessatsu}, pages 141--156. Res.
  Inst. Math. Sci. (RIMS), Kyoto, 2016.

\bibitem{mathieu2007quenched}
P.~Mathieu and A.~Piatnitski.
\newblock Quenched invariance principles for random walks on percolation
  clusters.
\newblock {\em Proc. R. Soc. Lond. Ser. A Math. Phys. Eng. Sci.},
  463(2085):2287--2307, 2007.

\bibitem{monrad1991nearby}
D.~Monrad and W.~Philipp.
\newblock Nearby variables with nearby conditional laws and a strong
  approximation theorem for {H}ilbert space valued martingales.
\newblock {\em Probab. Theory Related Fields}, 88(3):381--404, 1991.

\bibitem{morrow1982invariance}
G.~Morrow and W.~Philipp.
\newblock An almost sure invariance principle for {H}ilbert space valued
  martingales.
\newblock {\em Trans. Amer. Math. Soc.}, 273(1):231--251, 1982.

\bibitem{vardecay}
J.-C. Mourrat.
\newblock Variance decay for functionals of the environment viewed by the
  particle.
\newblock {\em Ann. Inst. Henri Poincar\'e Probab. Stat.}, 47(1):294--327,
  2011.

\bibitem{NS}
A.~Naddaf and T.~Spencer.
\newblock Estimates on the variance of some homogenization problems, 1998,
  {unpublished preprint}.

\bibitem{surveySkorokhod}
J.~Ob\l\'{o}j.
\newblock The {S}korokhod embedding problem and its offspring.
\newblock {\em Probab. Surv.}, 1:321--390, 2004.

\bibitem{varadhan1982}
G.~C. Papanicolaou and S.~R.~S. Varadhan.
\newblock Diffusions with random coefficients.
\newblock In {\em Statistics and probability: essays in honor of {C}. {R}.
  {R}ao}, pages 547--552. North-Holland, Amsterdam, 1982.

\bibitem{penrose1996large}
M.~D. Penrose and A.~Pisztora.
\newblock Large deviations for discrete and continuous percolation.
\newblock {\em Advances in applied probability}, 28(1):29--52, 1996.

\bibitem{pisztora1996surface}
A.~Pisztora.
\newblock Surface order large deviations for {I}sing, {P}otts and percolation
  models.
\newblock {\em Probab. Theory Related Fields}, 104(4):427--466, 1996.

\bibitem{procaccia2016quenched}
E.~Procaccia, R.~Rosenthal, and A.~Sapozhnikov.
\newblock Quenched invariance principle for simple random walk on clusters in
  correlated percolation models.
\newblock {\em Probab. Theory and Related Fields}, 166(3-4):619--657, 2016.

\bibitem{revuz2013continuous}
D.~Revuz and M.~Yor.
\newblock {\em Continuous martingales and Brownian motion}, volume 293.
\newblock Springer Science \& Business Media, 2013.

\bibitem{sakhanenko1984rate}
A.~I. Sakhanenko.
\newblock Rate of convergence in the invariance principle for variables with
  exponential moments that are not identically distributed.
\newblock In {\em Limit theorems for sums of random variables}, volume~3 of
  {\em Trudy Inst. Mat.}, pages 4--49. ``Nauka'' Sibirsk. Otdel., Novosibirsk,
  1984.

\bibitem{sapozhnikov2017random}
A.~Sapozhnikov.
\newblock Random walks on infinite percolation clusters in models with
  long-range correlations.
\newblock {\em Ann. Probab.}, 45(3):1842--1898, 2017.

\bibitem{shao1993almost}
Q.~M. Shao.
\newblock Almost sure invariance principles for mixing sequences of random
  variables.
\newblock {\em Stochastic Process. Appl.}, 48(2):319--334, 1993.

\bibitem{shao1995strong}
Q.~M. Shao.
\newblock Strong approximation theorems for independent random variables and
  their applications.
\newblock {\em J. Multivariate Anal.}, 52(1):107--130, 1995.

\bibitem{shorack1986empirical}
G.~R. Shorack and J.~A. Wellner.
\newblock {\em Empirical processes with applications to statistics}.
\newblock Wiley Series in Probability and Mathematical Statistics: Probability
  and Mathematical Statistics. John Wiley \& Sons, Inc., New York, 1986.

\bibitem{sidoravicius2004quenched}
V.~Sidoravicius and A.-S. Sznitman.
\newblock Quenched invariance principles for walks on clusters of percolation
  or among random conductances.
\newblock {\em Probab. Theory Related Fields}, 129(2):219--244, 2004.

\bibitem{skorohod1965}
A.~V. Skorokhod.
\newblock {\em Studies in the theory of random processes}.
\newblock Addison-Wesley Publishing Co., Inc., Reading, MA, 1965.
\newblock Translated from the Russian by Scripta Technica, Inc.

\bibitem{strassen}
V.~Strassen.
\newblock Almost sure behavior of sums of independent random variables and
  martingales.
\newblock In {\em Proc. {F}ifth {B}erkeley {S}ympos. {M}ath. {S}tatist. and
  {P}robability ({B}erkeley, {C}alif., 1965/66), {V}ol. {II}: {C}ontributions
  to {P}robability {T}heory, {P}art 1}, pages 315--343. Univ. California Press,
  Berkeley, CA, 1967.

\bibitem{Talagrand94}
M.~Talagrand.
\newblock The transportation cost from the uniform measure to the empirical
  measure in dimension {$\ge 3$}.
\newblock {\em Ann. Probab.}, 22(2):919--959, 1994.

\bibitem{Vershynin2018High}
R.~Vershynin.
\newblock {\em High-dimensional probability}, volume~47 of {\em Cambridge
  Series in Statistical and Probabilistic Mathematics}.
\newblock Cambridge University Press, Cambridge, 2018.
\newblock An introduction with applications in data science, With a foreword by
  Sara van de Geer.

\bibitem{Villani2009Optimal}
C.~Villani.
\newblock {\em Optimal transport}, volume 338 of {\em Grundlehren der
  mathematischen Wissenschaften [Fundamental Principles of Mathematical
  Sciences]}.
\newblock Springer-Verlag, Berlin, 2009.
\newblock Old and new.

\bibitem{zauitsev1996}
A.~Y. Zaitsev.
\newblock Estimates for the quantiles of smooth conditional distributions and
  the multidimensional invariance principle.
\newblock {\em Sibirsk. Mat. Zh.}, 37(4):807--831, ii, 1996.

\bibitem{zaitsev1998}
A.~Y. Zaitsev.
\newblock Multidimensional version of the results of {K}oml\'os, {M}ajor and
  {T}usn\'ady for vectors with finite exponential moments.
\newblock {\em ESAIM Probab. Statist.}, 2:41--108, 1998.

\bibitem{Zaitsev2002Estimates}
A.~Y. Zaitsev.
\newblock Estimates for the strong approximation in multidimensional central
  limit theorem.
\newblock In {\em Proceedings of the {I}nternational {C}ongress of
  {M}athematicians, {V}ol. {III} ({B}eijing, 2002)}, pages 107--116. Higher Ed.
  Press, Beijing, 2002.

\bibitem{zhai2018}
A.~Zhai.
\newblock A high-dimensional {CLT} in {$W_2$} distance with near optimal
  convergence rate.
\newblock {\em Probab. Theory Related Fields}, 170(3-4):821--845, 2018.

\bibitem{zhang1997strong}
L.-X. Zhang.
\newblock Strong approximation theorems for geometrically weighted random
  series and their applications.
\newblock {\em Ann. Probab.}, 25(4):1621--1635, 1997.

\end{thebibliography}

\end{document}